%% file: main.tex
\documentclass{article}
\input{header}

\title{Loss-aware distributionally robust optimization\\ via trainable \mfnew{optimal transport} ambiguity sets}

\author{%
  Jonas Ohnemus \\
  ETH Zürich \\
  \texttt{johnemus@ethz.ch} \\
  \And
  Marta Fochesato\\
  ETH Zürich\\
  \texttt{mfochesato@ethz.ch} \\
  \And
  Riccardo Zuliani \\
  ETH Zürich \\
  \texttt{rzuliani@ethz.ch} \\
  \And
  John Lygeros \\
  ETH Zürich \\
  \texttt{jlygeros@ethz.ch}
}

\begin{document}

\maketitle

\begin{abstract}
  \input{Contents/Abstract}
\end{abstract}

\section{Introduction}
\label{ch:intro}
\input{Contents/Introduction}

\section{Preliminaries}
\label{ch:preliminaries}
\input{Contents/Preliminaries}

\section{Learning the uncertainty in OT-DRO}
\label{ch:probform}

\input{Contents/Problem}

\section{Algorithm design}
\label{ch:algorithm}
\input{Contents/Algorithm}

\section{Numerical experiments}
\label{ch:numerics}
\input{Contents/Numerics}

\section{Concluding Remarks and Limitations}
\label{ch:conclusions}
\input{Contents/Conclusions}

\bibliography{references}


\include{Contents/Appendices}

\end{document}

%% file: header.tex
\PassOptionsToPackage{numbers, compress}{natbib}

\usepackage[preprint]{neurips_2025}

\usepackage[utf8]{inputenc} 
\usepackage[T1]{fontenc}    
\usepackage{hyperref}       
\usepackage{url}            
\usepackage{booktabs}       
\usepackage{amsfonts}       
\usepackage{nicefrac}       
\usepackage{microtype}      
\usepackage{xcolor}         
\usepackage{amsmath}
\usepackage{amssymb}
\usepackage{mathtools} 
\usepackage{amsthm}
\usepackage{physics}
\usepackage{graphicx}
\usepackage{dsfont}
\usepackage{wrapfig}
\usepackage{bbm}
\usepackage{algorithm2e}
\usepackage{xcolor}
\usepackage{subcaption}
\usepackage{adjustbox}
\usepackage{cleveref}
\crefname{theorem}{Theorem}{Theorems}
\crefname{lemma}{Lemma}{Lemmas}
\crefname{assumption}{Assumption}{Assumptions}
\crefname{subsection}{Section}{Sections}
\crefname{definition}{Definition}{Definitions}
\crefname{equation}{}{}
\crefname{proposition}{Proposition}{Propositions}
\crefname{figure}{Fig.}{Figs.}
\usepackage{enumitem}
\usepackage{tikz}
\usepackage{multirow}
\usepackage[page]{appendix}

\renewcommand{\appendixtocname}{Supplementary Material}

\makeatletter
\let\oldappendix\appendices

\renewcommand{\appendices}{%
  \clearpage
  \renewcommand{\thesection}{\Roman{section}}
  \let\tf@toc\tf@app
  \addtocontents{app}{\protect\setcounter{tocdepth}{1}}
  \immediate\write\@auxout{%
    \string\let\string\tf@toc\string\tf@app^^J
  }
  \oldappendix
}%

\newcommand{\listofappendices}{%
  \begingroup
  \renewcommand{\contentsname}{\appendixtocname}
  \makeatletter
  \renewcommand*\l@section{\@dottedtocline{1}{1.5em}{2.3em}}
  \makeatother
  \let\@oldstarttoc\@starttoc
  \def\@starttoc##1{\@oldstarttoc{app}}
  \tableofcontents
  \endgroup
}

\makeatother

\RestyleAlgo{ruled}

\definecolor{myred}{HTML}{CE2D2C}
\definecolor{myblue}{HTML}{3665BF}
\definecolor{mygreen}{HTML}{628B48}
\definecolor{myorange}{HTML}{FF8F00}
\definecolor{myviolet}{HTML}{9104AC}

\newtheorem{theorem}{Theorem}[section]

\newtheorem{remark}{Remark}[section]

\newtheorem{assumption}{Assumption}[section]
\newtheorem{proposition}{Proposition}[section] 
\newtheorem{lemma}[theorem]{Lemma} 
\newtheorem{definition}{Definition}
[section]
\newtheorem{fact}{Fact}[section]

\newcommand{\R}{\mathbb{R}}
\newcommand{\N}{\mathbb{N}}
\newcommand{\J}{\mathcal{J}}
\newcommand{\rz}[1]{#1}
\newcommand{\mf}[1]{#1}
\newcommand{\mfnew}[1]{#1}

\newcommand{\rzbegin}{}
\newcommand{\rzend}{}

\usetikzlibrary{shapes.misc}
\tikzset{cross/.style={cross out, draw=black, minimum size=2*(#1-\pgflinewidth), inner sep=0pt, outer sep=0pt},
cross/.default={1pt}}

\graphicspath{{figures/}}

\usepackage[most]{tcolorbox}
\usepackage{xcolor}

\definecolor{lightpurple}{RGB}{230, 210, 250}

%% file: Contents/Abstract.tex

{Optimal-Transport Distributionally Robust Optimization (OT-DRO) robustifies data-driven decision-making under uncertainty by capturing the sampling-induced statistical error via optimal transport ambiguity sets. The standard OT-DRO pipeline consists of a two-step procedure, where the ambiguity set is first designed and subsequently embedded into the downstream OT-DRO problem. However, this separation between uncertainty quantification and optimization might result in excessive conservatism. We introduce an end-to-end pipeline to automatically learn decision-focused ambiguity sets for OT-DRO problems, where the loss function informs the shape of the optimal transport ambiguity set, leading to less conservative yet distributionally robust decisions. We formulate the learning problem as a bilevel optimization program and solve it via a hypergradient-based method. By leveraging the recently introduced nonsmooth conservative implicit function theorem, we establish convergence to a critical point of the bilevel problem. We present experiments validating our method on standard portfolio optimization and linear regression tasks.}

%% file: Contents/Introduction.tex
Optimal Transport Distributionally robust optimization (OT-DRO) has recently emerged as a principled framework for decision-making under uncertainty due to its ability to capture distributional \rz{uncertainty} arising, for example, from sampling. The OT-DRO problem can be thought of as \rz{the} following zero-sum game
\begin{equation}
   \inf_{w\in\mathcal{W}} \sup_{\mathbb{Q} \in \mathcal{A}} \mathbb{E}_{\xi \sim \mathbb{Q}} \left[\ell(w,\xi) \right],\label{eq:decision_making_under_uncertainty}
\end{equation}
\rz{where the} decision-maker chooses a decision $w \in \mathcal{W} \subseteq \mathbb{R}^k$ to minimize the expectation of the loss function $\ell$, and \rz{the} adversary chooses a distribution $\mathbb{Q}$ from a so-called ambiguity set $\mathcal{A} \subseteq \mathcal{P}(\Xi)$ with $\mathcal{P}(\Xi)$ denoting the set of probability distributions supported on $\Xi \subseteq \mathbb{R}^d$.  
The OT-DRO problem in \eqref{eq:decision_making_under_uncertainty} has gained attention over the past years thanks to applications in machine learning \cite{shafieezadeh2015distributionally,blanchet2019robust,ho2023adversarial}, portfolio optimization \cite{blanchet2022distributionally}, control \cite{taskesen2023distributionally,aolaritei2023wasserstein}, and power systems \cite{poolla2020wasserstein}, among others. 

The traditional deployment of OT-DRO \eqref{eq:decision_making_under_uncertainty} relies on two steps carried out sequentially. In the first step, the ambiguity set is designed on the basis of the empirical observations, typically in the form of independent samples $\{\hat{\xi}_1, \ldots, \hat{\xi}_J\}$    extracted from an unknown probability distribution $\mathbb{P}$. \rz{A} natural way to integrate this information is by \rz{constructing} an ambiguity set 
\begin{equation}
    \mathcal{A} := \mathcal{B}_\varepsilon(\hat{\mathbb{P}}) = \left\{ \mathbb{Q} \in \mathcal{P}(\Xi) \; \vert \; d(\mathbb{Q}, \hat{\mathbb{P}}) \leq \varepsilon \right\},\label{eq:ambiguity_set}
 \end{equation}
\rz{defined as} a ball of radius $\varepsilon \in \mathbb{R}_{\geq 0}$ centered around the empirical distribution $\hat{\mathbb{P}} = \frac{1}{J} \textstyle \sum_{j=1}^J \delta_{\hat{\xi}_j}$. Here, $d:\mathcal{P}(\Xi) \times \mathcal{P}(\Xi) \rightarrow [0,+\infty)$ is an optimal transport based discrepancy defined as $d (\mathbb{P}, \hat{\mathbb{P}}) := \inf_{\pi \in \Pi(\mathbb{P}, \hat{\mathbb{P}})} \mathbb{E}_{\xi_1, \xi_2 \sim \pi} [\kappa(\xi_1, \xi_2)]$, where $\kappa:\Xi \times \Xi \rightarrow [0, +\infty)$ is a prescribed transportation cost function satisfying the identity of indiscernibles, and $\Pi(\mathbb{P}, \hat{\mathbb{P}})$ represents the set of all joint probability distributions of $\xi_1, \xi_2$ with marginals $\mathbb{P}, \hat{\mathbb{P}}$, respectively. In the second step, $\mathcal{A}$ is embedded into \eqref{eq:decision_making_under_uncertainty} and the resulting problem is solved either by resorting to finite-dimensional reformulations grounded on duality theory \cite{gao2023distributionally,blanchet2019quantifying}, or on stochastic gradient methods \cite{yu2022fast,li2019first,li2020fast}. If the ambiguity set is designed \rz{to} contain the true distribution $\mathbb{P}$ with high probability, DRO theory offers an out-of-sample certificate ensuring that when deploying the obtained minimizer on new, unseen data, \textit{post-decision disappointment} \rz{does not occur} with high probability (unlike standard empirical risk minimization methods, such as sample average approximation \cite{smith2006optimizer}). 



This traditional OT-DRO pipeline can result in an overly conservative decision, as the downstream optimization problem does not inform the geometry of the ambiguity set. Consider, for example, the choice of the transportation cost as $\kappa(\xi_1, \xi_2) = \|\xi_1 - \xi_2\|^p_2$, resulting in the celebrated type-p Wasserstein distance \cite{esfahani2017datadrivendistributionallyrobustoptimization}. Intuitively, this choice requires that the true distribution deviates little from the empirical estimate in \textit{all} directions. In reality, however, our primary concern is to exclude only those distributions that actively contribute to increasing the worst-case cost, rather than imposing uniform constraints in all directions. This intuition motivates the introduction of a new OT-DRO methodology, where the geometry of the ambiguity set --- which represents a degree of freedom that has not been exploited thus far --- is informed by the loss function $\ell$ to ensure coverage of the true distribution while selectively excluding adversarial distributions that inflate the worst-case cost.


\begin{tcolorbox}[colback=lightpurple!25, colframe=white, boxrule=0pt, arc=2pt, left=4pt, right=4pt, top=4pt, bottom=4pt]

Our contributions are as follows. 

\begin{itemize}[left=0pt]

{
\item \textbf{End-to-end OT-DRO pipeline.} We introduce a novel \textit{end-to-end} OT-DRO pipeline to automatically learn decision-focused ambiguity sets $\mathcal{A}_\theta$, defined by some parameter $\theta \in \mathbb{R}^{n_\theta}$, leading to less conservative, yet distributionally robust solutions. Specifically, among all ambiguity sets leading to the same out-of-sample disappointment $\beta$, we want to determine the one ensuring minimum out-of-sample risk. 


\smallskip

\item \textbf{Algorithmic solution.} We encode the problem as a bilevel optimization program where the upper level chooses a geometry for $\mathcal{A}_\theta$, for example via a parametrized transportation cost $\kappa(\cdot, \cdot; \theta)$, while the lower level solves \eqref{eq:decision_making_under_uncertainty} for the chosen ambiguity set $\mathcal{A}_\theta$.
To solve the bilevel problem, we use a hypergradient-based method based on the recently introduced nonsmooth conservative implicit function theorem, and we show that, under mild conditions, our numerical scheme provably converges to the set of critical points of the bilevel program (or to a neighbourhood of them) despite the nonsmoothness and nonconvexity of the solution map of the lower level.

\smallskip

\item \mfnew{\textbf{Software.} We present an open-source implementation of our algorithm in Python. We make our code available at: \url{https://github.com/JonasOhn/trainable-ot-dro}.}}

\end{itemize}
\end{tcolorbox}

\section{Related works}

\textbf{Metric learning}.  
%

\mfnew{Metric learning \cite{bellet2013survey} (also referred to as "smart predict-and-optimize" in operations research \cite{elmachtoub2022smart}) refers to the paradigm of training a predictive model to minimize the 
loss on a downstream optimization task and has gained increasing attention across several domains \cite{donti2017task}, \cite{demirovic2019investigation}, \cite{wilder2019melding}, \cite{cameron2022perils} (see also the survey papers \cite{mandi2024decision,sadana2025survey}). Recently, it has also been applied to design decision-focused uncertainty sets \cite{chenreddy2022data,sun2023predict,wang2023learning,chenreddy2024end}.}
\rzbegin
In particular, \cite{wang2023learning} addresses contextual stochastic optimization, where the goal is to learn an uncertainty set (in $\mathbb{R}^n$) that maximizes expected performance across a family of contextual problems. By contrast, we consider distributional ambiguity and robustify in the space of probability distributions.
Within OT-based DRO, \cite{blanchet2019datadrivenoptimaltransportcostselectionfordro} calibrates a Mahalanobis distance to penalize directions with high performance impact. However, their approach is limited to linear regression and requires a separate calibration step, in contrast to our end-to-end framework, where the DRO loss directly informs the set design to reduce conservatism. \cite{behzadian2019reshapingambiguitysetsinrobustmdps} learns norm weights in robust Markov decision processes with finite state-action spaces, while \cite{schuurmans2023distributionallyrobustoptimizationusing} selectively enlarges ambiguity sets of discrete distributions in directions with limited effect on the worst-case cost. However, neither approach extends to continuous distributions or supports flexible parameterizations of the transport cost. Further, \cite{costa2023distributionally} leverages residual structure of the uncertainty distribution, limiting general applicability, and \cite{ma2024differentiable} differentiates through conic programs to learn conic-representable sets for mixed-integer DRO, but overlooks the nondifferentiability of the solution map. Finally, compared to \cite{chaouach2023structured}, our approach reduces conservatism without relying on independence assumptions among features. 


\rzend

\textbf{Differentiable optimization.} Differentiable optimization refers to the practice of differentiating the solution map of optimization problems, generally by applying the implicit function theorem to their optimality conditions \cite{dontchev2009implicit}. This idea has been applied to differentiate the solution map of quadratic programs \cite{amos2017optnet}, linear programs with a regularizing term \cite{mandi2020interior}, and linear conic programs \cite{busseti2019solution,agrawal2019differentiating}. As a direct consequence of the implicit function theorem, all these methods implicitly assume continuous differentiability of the solution map with respect to the problem parameters. However, the solution map of an (even convex) optimization program is generally not everywhere differentiable. To relax the continuous differentiability assumption, one can utilize the concept of conservative Jacobians \cite{bolte2021conservative}, which extend traditional gradients to almost everywhere differentiable functions, and, most notably, admit a nonsmooth implicit function theorem \cite{bolte2021nonsmooth}. Leveraging this concept, \cite{zuliani2025bp} develops a first-order method with convergence guarantees to solve a bilevel problem with a quadratic lower level. In \cite[Proposition 4]{bolte2021nonsmooth}, the authors apply the nonsmooth implicit function theorem in the context of conic programs. This work is a fundamental building block for the algorithm used in this paper.



%% file: Contents/Preliminaries.tex
\textbf{Notation.} We assume an underlying probability space $(\Omega, \mathcal{F}, P)$ and define the distribution of any random vector $\xi:\Omega \rightarrow \mathbb{R}^d$ by the pushforward distribution $\mathbb{P} = P \circ \xi^{-1}$ of $P$ with respect to $\xi$. \mf{$\mathcal{P}(\Xi)$ denotes the set of probability distributions on $\Xi \subseteq \mathbb{R}^d$ and $\mathcal{P}_g(\Xi)$ its restriction to the set of Gaussians. We use $\mathcal{N}(\mu, \Sigma)$ to denote a Gaussian distribution with mean $\mu \in \mathbb{R}^d$ and covariance $\Sigma \in \mathbb{R}^{ d\times d}$, and $\mathcal{U}(a,b)$ to denote a uniform distribution in the interval $[a,b]$.} For $n \in \mathbb{Z}_+$, we set $[n] = \{1,\ldots, n\}$. \rz{Given a probability distribution $\mathbb{P}$} and a set $\mathcal{X}$, \mf{we use $\mathbb{P}^n \coloneqq \mathbb{P} \times \ldots \times \mathbb{P}$ and $\mathcal{X}^{\otimes n} \coloneqq \mathcal{X} \otimes \ldots \otimes \mathcal{X}$ to denote the product distribution and the product set, respectively. We denote the Euclidean norm with $\|\cdot\|$ and use $\text{dist}(x, \mathcal{X}) \coloneqq \inf \{ \| x - z\| \; | \; z \in \mathcal{X} \}$ to denote the point-to-set distance. Lastly, $\mathbb{L}_{++}^n$ denotes the set of positive definite lower triangular matrices.}

\textbf{Path differentiability.} \emph{Conservative Jacobians} can be used to generalize the notion of Jacobian to functions that are almost everywhere differentiable \cite{bolte2021conservative}. Specifically, given a locally Lipschitz function $f:\R^n\to\R^m$, we say that the outer semicontinuous, compact-valued map $\J_f:\R^n\rightrightarrows \R^{m \times n}$ is a conservative Jacobian of $f$ if, given any absolutely continuous function $\varphi:[0,1]\to \R^n$, $\frac{d}{dt}f(\varphi(t))= V\dot{\varphi}(t)$ for any $V\in \J_f(\mfnew{\varphi(t)})$ and almost every $t\in[0,1]$. By Rademacher's theorem, $\nabla f(x)$ exists for almost every $x\in\R^n$, in which case, by \cite[Theorem 1]{bolte2021conservative}, $\J_f(x)=\{\nabla f(x)\}$, meaning that $\J_f$ coincides almost everywhere with the standard Jacobian. We say that a function $f$ is \emph{path-differentiable} if it admits a conservative Jacobian. Here, we focus on the class of locally Lipschitz functions that are \textit{definable in an o-minimal structure}, or simply \textit{definable} \cite{coste1999introduction} (\mfnew{see also Appendix \ref{app:definability} for a concise explanation}), which always admit a conservative Jacobian \cite[Proposition 2]{bolte2021conservative}. This class of functions contains most functions commonly found in the fields of control and optimization, including semialgebraic functions and analytic functions restricted to a definable domain.

Given a locally Lipschitz definable function $f:\R^n\to\R$ and a sequence of positive step sizes $\{\alpha_i\}_{i\in\N}$, the update rule
\begin{equation}
    x_{i+1} = x_i-\alpha_id_i, ~~d_i\in \J_f(x_i),
\end{equation}
is guaranteed to converge to a critical point $\bar{x}$ for which $0\in \J_f(\bar{x})$ if $\alpha_i>0$ is square summable but not summable \cite[Theorem 3.2]{davis2020stochastic}. If $\alpha_i\equiv\bar{\alpha}$, then for a small enough $\bar{\alpha}$, $\limsup_{i\to \infty} \operatorname{dist}(0,\J_f(x_i))\leq \epsilon$, where $\epsilon>0$ can be made arbitrarily small by reducing $\bar{\alpha}$ \cite[Theorem 2]{bolte2024inexact}.


\textbf{Differentiating through conic programs.} Conic programs are a broad class of optimization problems of the form
\begin{equation}\label{standard:conic}
    \begin{aligned}
        \min_{x,s} \quad & {c}^\top x \\
        \textrm{subject to} \quad & {A}x + s = {b} \\
        & (x, s) \in \mathbb{R}^n \times \mathcal{K},
    \end{aligned}
\end{equation}
where $\mathcal{K}$ is a closed convex cone such as the nonnegative orthant or \rz{the} second-order cone. Assuming $(A,b,c)=(A(\theta),b(\theta),c(\theta))$ depend on a parameter $\theta\in\R^{n_\theta}$, one can treat \rz{the primal-dual solution $(x^*(\theta),y^*(\theta),s^*(\theta))$ of \eqref{standard:conic} (assuming its existence)} as a function of $\theta$ and write the solution map as $\mathcal{S}:\R^{n_\theta}\to\R^n$, where \rz{$\mathcal{S}(\theta)=(x^*(\theta),y^*(\theta),s^*(\theta))$}. By implicitly differentiating the KKT conditions of \eqref{standard:conic}, it is possible to obtain the conservative Jacobian $\J_\mathcal{S}(\theta)$ of $\mathcal{S}$ \cite{bolte2021nonsmooth}. A more thorough description of the differentiation procedure is provided in Appendix~\ref{app:diff_through_conic_programs}.
%
    

%% file: Contents/Problem.tex

\subsection{Problem formulation}
\rz{We modify the OT-DRO problem \cref{eq:decision_making_under_uncertainty} by replacing $\mathcal{A}$ with a parametrized ambiguity set $\mathcal{A}_\theta$ given by}
\begin{equation}\label{eq:ambiguity:set:parametrized}
\rz{\mathcal{A}_\theta := \mathcal{B}_\varepsilon(\hat{\mathbb{P}};\theta) := \{ \mathbb{Q} \in \mathcal{P}(\Xi) \: |\: d(\mathbb{Q}, \hat{\mathbb{P}};\theta) \leq \varepsilon\},}
\end{equation}
\rz{where the parameter $\theta \in \Theta \subseteq \mathbb{R}^{n_\theta}$ affects the optimal transport based discrepancy $d(\mathbb{Q}, \hat{\mathbb{P}};\theta):=\inf_{\pi \in \Pi(\mathbb{P}, \hat{\mathbb{P}})} \mathbb{E}_{\xi_1, \xi_2 \sim \pi} [\kappa(\xi_1, \xi_2; \theta)]$ through the transportation cost $\kappa(\cdot,\cdot;\theta)$, thus defining the geometry of the set.} This leads to the parameterized OT-DRO problem
\begin{equation}
   \inf_{w\in\mathcal{W}} \sup_{\mathbb{Q} \in \mathcal{A}_\theta} \mathbb{E}_{\xi \sim \mathbb{Q}} \left[\ell(w,\xi) \right].\label{eq:decision_making_under_uncertainty_parameterized}
\end{equation}
\mf{In the remainder, we always assume that \eqref{eq:decision_making_under_uncertainty_parameterized} is well-posed for all $\theta \in \Theta$, that is, it admits a finite minimizer $\hat{w}_\theta$. Conditions for well-posedness of problem \eqref{eq:decision_making_under_uncertainty_parameterized} have been established in \cite{yue2022linear}.} Let $\hat{\ell}(\hat{w}_\theta)$ be the corresponding optimal solution, and let $\ell^\star :=\inf_{w \in \mathcal{W}}\mathbb{E}_{\xi \sim\mathbb{P}}[\ell(w, \xi)]$ be the optimal loss under complete knowledge of the distribution $\mathbb{P}$. Our goal is to determine a parameter vector $\theta^\star$ \rz{attaining} to the lowest in-sample loss $\hat{\ell}(\hat{w}_{\theta^\star})$ satisfying
\begin{equation}\label{eq:certificate}
\text{Pr}(\ell^\star \leq  \mathbb{E}_{\xi \sim\mathbb{P}}[\ell(\hat{w}_{\theta^\star}, \xi)] \leq \hat{\ell}(\hat{w}_{\theta^\star})) \geq 1 - \beta,
\end{equation}
for a given user-defined reliability parameter $\beta \in (0,1)$. Equation \eqref{eq:certificate} represents an out-of-sample performance certificate on the data-driven decision $\hat{w}_{\theta^\star}$. By \rz{maximally reducing} $\hat{\ell}(\hat{w}_{\theta^\star})$ \rz{through} a careful design of the geometry of $\mathcal{A}_\theta$, the out-of-sample loss $\mathbb{E}_{\xi \sim\mathbb{P}}[\ell(\hat{w}_{\theta^\star}, \xi)]$ \rz{gradually approaches} the true loss $\ell^\star$, which represents a fixed problem-specific global lower bound.

\rz{Since} \eqref{eq:certificate} is implied by the condition $\text{Pr}(\mathbb{P} \in \mathcal{B}_\epsilon(\hat{\mathbb{P}};\theta^\star)) \geq 1 - \beta$, we can formalize the problem as
\begin{subequations}\label{eq:formulation:1}
    \begin{align}\label{eq:cost}
       \hat{\ell}({w}_{\theta}) :=   \inf_{w\in\mathcal{W}, \, \theta\in \Theta} \sup_{\mathbb{Q} \in \mathcal{B}_\varepsilon(\hat{\mathbb{P}};\theta)}\quad & \mathbb{E}_{\xi \sim \mathbb{Q}} \left[ \ell(w, \xi) \right] \\ \label{constraint}
         \textrm{subject to} \quad &\mathrm{Pr}(\mathbb{P} \in \mathcal{B}_\varepsilon(\hat{\mathbb{P}};\theta)) \geq 1-\beta.
    \end{align}
    \label{eq:dream_opt_problem}
\end{subequations}

\begin{wrapfigure}{r}{0.3\textwidth}
    \vspace{-1.5em}
    \centering
    \input{figures/sketch.tex}
    \caption{Simplified problem.}
    \label{fig:ambset_reshaping_sketch}
    \vspace{-2.5em}
\end{wrapfigure}
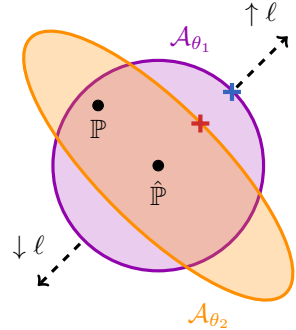

The interpretation is that among all the ambiguity sets containing the true probability distribution $\mathbb{P}$ with \rz{high likelihood}, we select the one resulting in the lowest value of $\hat{\ell}$.

\rz{As an example, consider the simplified problem in Figure~\ref{fig:ambset_reshaping_sketch}, where the loss increases in the north-east direction. If the orange and purple ambiguity sets both contain $\mathbb{P}$ with probability $1-\beta$, then choosing the orange set is more convenient, as this set produces a less conservative solution (red cross) compared to the purple set (blue cross) while maintaining the same out-of-sample guarantees. }

\rz{We propose an automated pipeline to solve \eqref{eq:formulation:1} that solely relies on the structure of the loss function $\ell$ and on the availability of $J$ samples $\{\hat{\xi}_1, \ldots, \hat{\xi}_J\}$ without requiring knowledge of $\mathbb{P}$.}


\subsection{A bilevel formulation}
\label{ch:bilevelopt}
We consider a bilevel surrogate of \eqref{eq:formulation:1} \rz{that exploits} the structural properties of the problem
\begin{equation}\label{eq:init_bilevel_opt_problem}
    \begin{aligned}
        \inf_{\theta \in \Theta}  \sup_{\mathbb{Q} \in \mathcal{B}_\varepsilon(\hat{\mathbb{P}};\theta)} & \mathbb{E}_{\xi \sim \mathbb{Q}} \left[ \ell(w_\theta^\star, \xi) \right]  \\
        \textrm{subject to} \quad & \mathrm{Pr}(d(\mathbb{P}, \hat{\mathbb{P}}; \theta) \leq \varepsilon) \geq 1-\beta\\
        & w^\star_\theta = \arg\inf_{w\in\mathcal{W}} \sup_{\mathbb{Q} \in \mathcal{B}_\varepsilon(\hat{\mathbb{P}};\theta) } \mathbb{E}_{\xi \sim \mathbb{Q}} \left[ \ell(w, \xi) \right].
    \end{aligned}
\end{equation}

Compared to \eqref{eq:dream_opt_problem}, the bilevel formulation in \eqref{eq:init_bilevel_opt_problem} \rz{separates} the joint minimization over $w$ and $\theta$ across two different levels, with the shape-inducing variable $\theta$ being optimized at the upper level and the task-related variable $w$ being optimized at the lower level. We can think of \eqref{eq:init_bilevel_opt_problem} as the problem of \textit{tuning} an optimal transport-based ambiguity set.

\textbf{Lower level.} For a fixed $\theta \in \Theta$, the lower-level constitutes a standard OT-DRO problem of the form \eqref{eq:decision_making_under_uncertainty}. Under mild regularity conditions listed in Appendix \ref{app:dro_cp_reformulation}, \eqref{eq:decision_making_under_uncertainty} admits a finite-dimensional convex reformulation
\begin{align}\label{eq:conic:problem:DRO}
    \mathcal{S}\rz{(\theta)} = \begin{cases}
        \arg\min &c(\theta)^\top x\\
        \mathrm{s.t.} &A(\theta)\;x + s = b(\theta)\\
        &s \in \mathcal{K}.
    \end{cases}
\end{align}
where $\mathcal{S}(\theta) = (x^\star(\theta),y^\star(\theta),s^\star(\theta))$ in the solution map of \cref{eq:conic:problem:DRO} and it groups the primal variable, the dual variable, and the slack variable, respectively, while $\mathcal{K}$ is a convex and closed cone. The original variable $w$ is a block entry of the primal variable $x$, and can be easily extracted from it. Problem \eqref{eq:conic:problem:DRO} is a parametrized conic program, where the functional representation of $(A(\theta), b(\theta), c(\theta))$ depends on the specific parametrization of the transportation cost $\kappa(\cdot, \cdot;\theta)$. \mf{We rely on the following standing assumption.}
\begin{assumption}\label{assumption:unique}
\mf{The minimizer of the lower level of \eqref{eq:init_bilevel_opt_problem} is unique for all $\theta \in \Theta$.}
\end{assumption}
\mfnew{Assumption \ref{assumption:unique} is typically employed in the context of implicit differentiation, see e.g. \cite{bolte2021nonsmooth}.}


Examples of suitable parameterizations $\kappa(\cdot,\cdot;\theta)$ of the transportation cost that lead to formulations of the form \eqref{eq:conic:problem:DRO} include:
\begin{enumerate}
\item \emph{Mahalanobis distance}: Let $L \in \mathbb{L}_{++}^d$ be a positive definite lower triangular matrix, and $p \in [0,\infty)$. Then, for $\theta = L$, the transportation cost $\kappa(\xi_1, \xi_2; \theta) = \norm{L^\top (\xi_1 - \xi_2)}_2^p$ encode anisotropic features in the form of different sensitivities along different directions in $\mathbb{R}^d$. \mfnew{More specifically, using the singular value decomposition of $L = U \Pi V^\top$, with $\{\pi_i\}_{i=1}^d$ being the singular values of $L$ and $\{v_i\}_{i=1}^d$ the orthonormal columns of $V$, the transportation cost becomes $\|L(\xi_1 - \xi_2)\|^p_2 = \left({\sum_{i=1}^d \pi_i^2|v_i^\top(\xi_1 - \xi_2)|^2}\right)^{p/2}$. Thus, moving probability mass from the center distribution in the direction $v_i$  costs $\pi_i |\xi_1 - \xi_2|$: the higher the value of $\pi_i$ , the less probability mass is moved in the direction $v_i$ and vice versa. }
\item \emph{Functional composition of norms}: Let $f: \mathbb{R}_{\geq 0} \times \R^d \rightarrow \mathbb{R}_{\geq 0}$ be a continuous and strictly convex function parametrized for any $\theta \in \mathbb{R}^d$ (i) $x_1 \leq x_2 \implies f(x_1;\theta) \leq f(x_2;\theta)$, (ii) $\lim_{x \rightarrow \infty} \frac{f(x;\theta)}{x} \rightarrow \infty$, (iii) $f(0;\theta) = 0$. Then, $f(\|\xi_1 - \xi_2\|;\theta)$ is a valid parametrization. Possible choices of $f$ include convex combinations of norms, the (scaled) exponential and logarithmic functions, and the maximum of quadratic functions.
\item \emph{\mfnew{Conic} combination of transportation costs}: Let $\kappa_1(\xi_1, \xi_2),\ldots, \kappa_m(\xi_1,\xi_2)$ be valid transportation costs; then $\sum_{i=1}^m \theta_i \kappa_i(\xi_1, \xi_2)$, \mfnew{with $\theta_i \geq 0$}, is a valid parametrization. \mfnew{If the dictionary $\{\kappa_i\}_{i=1}^m$ is complete\footnote{\mfnew{By complete we mean that the dictionary spans the entire space of admissible transportation costs defined as $\mathcal{T} \coloneqq \{ \kappa(x,y) \: \text{symmetric}, \kappa(x,x) = 0, k(x,y) > 0 \: \text{iff} \: x \neq y\}$. For example, if the dictionary only includes Euclidean distances possibly raised to some powers, their conic hull cannot approximate non-metric costs. Nonetheless, even non-complete dictionaries are often "rich" enough in practice.}}, this parametrization universally determines all possible geometries of the ambiguity set.}
\end{enumerate}

\textbf{Upper level.} The upper-level problem is given by
\begin{subequations}\label{eq:bilevel_problem_unknown_dist}
    \begin{align}\label{cost:upper}
        \min_{\theta \in \Theta} \quad & c(\theta)^\top x^\star\left(\theta\right) \\
        \textrm{subject to} \label{constraint:upper}
        \quad &\mathrm{Pr}(d(\mathbb{P}, \hat{\mathbb{P}}; \theta) \leq \varepsilon) \geq 1-\beta.
    \end{align}
\end{subequations}
\rz{The probability in \eqref{constraint:upper} is taken with respect to the dataset $\hat{\mathcal{D}}_J := \{\hat{\xi}_1, \ldots, \hat{\xi}_J\}$ used to construct the reference distribution $\hat{\mathbb{P}}:= \hat{\mathbb{P}}(\hat{\mathcal{D}}_J) = \frac{1}{J} \textstyle \sum_{j=1}^J \delta_{\hat{\xi}_j}$. The set $\hat{\mathcal{D}}_J$ is a realization of the random multi-sample $\mathcal{D}_J$ distributed according to $\mathbb{P}^{J}$ and supported on $(\Xi)^{\otimes J}$.}

\rz{Note that \eqref{constraint:upper} depends on the true distribution $\mathbb{P}$, which is not known in our setting. We can approximate the probability in \eqref{constraint:upper} by bootstrapping samples of $\mathcal{D}_J$ from the set $\hat{\mathcal{D}}_J$ with replacement, obtaining}
\begin{equation}
\begin{aligned}
    \mathrm{Pr}(d(\mathbb{P},\hat{\mathbb{P}};\theta) \leq \varepsilon) & \approx \mathop{\mathbb{E}}_{{\mathcal{D}}_J\sim \mathbb{P}^{ J}}\left[ \mathds{1}\left\{ d(\mathbb{P}, \hat{\mathbb{P}}({\mathcal{D}}_J);\theta) \leq \varepsilon \right\} \right]\\
   &  \approx \frac{1}{n_b} \sum_{k=1}^{n_b} \mathds{1}\left\{ d(\hat{\mathbb{P}}, \hat{\mathbb{P}}({\hat{\mathcal{D}}}_J^k);\theta) \leq \varepsilon \right\}.
\end{aligned}
\end{equation}
where $n_b$ denotes the number of multi-samples $\hat{\mathcal{D}}_J^k$ extracted from \mf{$\hat{\mathbb{P}}$}. \rz{We can then rewrite \eqref{constraint:upper} as}
\begin{equation}\label{eq:constraint:boot}
\frac{1}{n_b}\sum_{k=1}^{n_b} \mathds{1}\left\{ d(\hat{\mathbb{P}}, \hat{\mathbb{P}}(\hat{\mathcal{D}}_J^k); \theta) \leq \varepsilon \right\} \geq 1 - \beta,
\end{equation}
\rz{where we require the distance between the nominal distribution and at least a $1-\beta$ fraction of the bootstrapped distributions to not exceed $\varepsilon$. Notice that \eqref{eq:constraint:boot} can effectively be implemented with the available information.}

%% file: figures/sketch.tex
\begin{tikzpicture}[scale=0.8]  

    \draw[very thick, myviolet, fill, fill opacity = 0.2] (0,0) circle (1.75);

    \draw[very thick, myorange, rotate around={45:(0,0)}, fill, fill opacity = 0.275] (0,0) ellipse (1 and 3);

    \filldraw (0,0) circle (2.5pt);
    \node[below,yshift=-2pt] at (0,0) {$\hat{\mathbb{P}}$};

    \filldraw (-1,1) circle (2.5pt);
    \node[below,yshift=-2pt] at (-1,1) {$\mathbb{P}$};

    \node[myviolet] at (0.55,2.1) {$\mathcal{A}_{\theta_1}$};
    \node[myorange] at (0.85,-2.5) {$\mathcal{A}_{\theta_2}$};

    \draw[->, dashed, very thick] (1.3,1.3) -- node[pos=1,above right, xshift=-20pt] {$\uparrow \ell$} (2.15,2.15);
    \draw[->, dashed, very thick] (-1.3,-1.3) -- node[pos=0.6,above left] {$\downarrow \ell$} (-2,-2);

    \draw (0.7,0.7) node[ultra thick,cross=4pt,rotate=45,myred] {};
    \draw (1.225,1.225) node[ultra thick,cross=4pt,rotate=45,myblue] {};

\end{tikzpicture}

%% file: Contents/Algorithm.tex

In this section, we devise a hypergradient-based algorithm with convergence guarantees to solve the problem in \eqref{eq:init_bilevel_opt_problem}. 
The main challenges lie in (i) ensuring that $\theta$ belongs to the feasible set $\Theta := \{ \theta \in \mathbb{R}^p\: : \: \eqref{eq:constraint:boot} \: \text{holds}\}$, and (ii) obtaining the hypergradient, i.e., the gradient of the upper-level objective function $c(\theta)^\top \mathcal{S}(\theta)$ with respect to $\theta$.

Generally, $\Theta$ is a nonconvex set as the condition in \eqref{eq:constraint:boot} is not convex in $\theta$ (see Appendix \ref{ch:convexity_param_discrete_wassdist} for a proof in the case of the Mahalanobis distance). To avoid computationally expensive projection operations, we solve instead the following unconstrained optimization problem, where \rz{the constraints are replaced with a penalty function}
\begin{equation}
    \begin{aligned}
        \min_{\theta}~~& \varphi (\theta) := \underbrace{c(\theta)^\top x^\star(\theta)}_{:=\varphi_\text{o}(\theta, x^\star(\theta))} + \underbrace{\lambda_\mathrm{p} \max \{0,e(\theta)\}^2}_{:=\varphi_\text{p}(\theta)}. \label{eq:bilevel_solvable}
    \end{aligned}
\end{equation}
where
\begin{equation}\label{eq:penalty}
    e(\theta) = \left( \frac{1}{n_b} \sum_{k=1}^{n_b}\sigma \left(d(\hat{\mathbb{P}}, \hat{\mathbb{P}}(\hat{\mathcal{D}}_J^k); \theta)/\varepsilon - 1 \right) \right) - \beta,
\end{equation}
and $ \sigma(x) = [1+\exp\left( -\eta_\mathrm{p} x \right)]^{-1}$ is the sigmoid function. In \cref{eq:bilevel_solvable}, $\varphi_\text{p}(\theta)$ penalizes positive constraint violations $e(\theta)$ by a large coefficient $\lambda_p \in \mathbb{R}_{> 0}$. The expression in \cref{eq:penalty} is a smooth approximation of the indicator function in \eqref{eq:constraint:boot}, \rz{with $\eta_\mathrm{p} \in \mathbb{R}_{> 0}$ regulating the approximation accuracy---larger values yielding a better approximation}.

To obtain the hypergradient of the objective in \cref{eq:bilevel_solvable} we need the following assumption.
\begin{assumption}\label{ass:path_diff_bilevel}
The functions $c$ and $x^\star$ are locally Lipschitz and definable in $\theta$. For any ${\mathbb{P}}_1$ and ${\mathbb{P}}_2$, $d({\mathbb{P}}_1,{\mathbb{P}}_2;\theta)$ is locally Lipschitz and definable in $\theta$.
\end{assumption}
\cref{ass:path_diff_bilevel} is mild and has already been proposed and studied in the context of conic programming \cite{bolte2021nonsmooth,bolte2024differentiating}. In Appendix~\ref{app:diff_through_conic_programs} we provide more details outlining sufficient conditions under which \cref{ass:path_diff_bilevel} holds in our setting. Under \cref{ass:path_diff_bilevel}, the hypergradient of \eqref{eq:bilevel_solvable} can be obtained by applying the chain rule
\begin{align}
\J_\varphi(\theta) = \left\{ J_{\varphi_\text{o}} + 2\lambda_\mathrm{p} \max \{0,e(\theta)\} J_e: J_{\varphi_\text{o}}\in\J_{\varphi_o}(\theta),~ J_e\in\J_{e}(\theta) \right\}, \label{eq:chain_rule_1}
\end{align}
where 
\begin{equation}\J_{\varphi_\text{o}}(\theta) = \{ J_{\varphi_\text{o},\theta} + J_{x^\star}^{\top}J_{\varphi_\text{o},x}: [J_{\varphi_\text{o},\theta}~J_{\varphi_\text{o},x}]\in\J_{\varphi_\text{o}}(\theta,x^\star(\theta)),~J_{x^\star}\in\J_{x^\star}(\theta)\}.
\end{equation}
Computing the Jacobian of the constraint violation $e(\theta)$ requires differentiating the value function of an optimal transport problem with respect to $\theta$, as explained in Appendix~\ref{app:diff_ot}. The conservative Jacobian $\J_{x^\star}$ of $x^\star$ can be obtained by differentiating the solution of the conic program \cref{eq:conic:problem:DRO}, as explained in Appendix~\ref{app:diff_through_conic_programs}.

Given an element $J_{\varphi_\text{o}}(\theta_i)$ of $\J_{\varphi_\text{o}}(\theta_i)$, the update step for $\theta$ follows the hypergradient descent dynamics $\theta_{i+1}=\theta_i-\alpha_i J_{\varphi_\text{o}}(\theta_i)$, 
where $\alpha_i \in \mathbb{R}_{> 0}$. 
%
We summarize the proposed learning procedure in Algorithm~\ref{alg:learning_ambiguity}.

\begin{algorithm}[H]
    \renewcommand{\baselinestretch}{1.2}\selectfont
    \DontPrintSemicolon
    \SetAlgoLined
    \KwIn{initial guess $\theta_0$, samples $\hat{\mathcal{D}}_J =\{\rz{\hat{\xi}_j}\}_{j=1}^J$, step sizes $\{\alpha_i\}_{i\in\mathbb{N}}, \alpha_i>0$}
    \KwOut{$\theta^\star$, $\hat{w}_{\theta^\star}$}
    Bootstrap $n_b$ reference distributions $\{\hat{\mathbb{P}}_k\}_{k=1}^{n_b}$ by resampling $\hat{\mathcal{D}}_J$ with replacement\\
    Find $\varepsilon$ as the $1-\beta$ quantile of $\{d(\hat{\mathbb{P}}_k, \hat{\mathbb{P}}; \theta_0)\}_{k=1}^{n_b}$\\
    \For{$i \leftarrow 1$ \KwTo {\normalfont\texttt{maxiter}}}{
    Solve conic program with $A(\theta_i), b(\theta_i), c(\theta_i)$ for \rz{$z^\star(\theta_i)$}\\
    Differentiate through conic program for $J_\mathcal{S}(\theta_i)^\top c(\theta_i)$ \\
    Calculate $d(\hat{\mathbb{P}}_k, \hat{\mathbb{P}};\theta_i)$ for all $k\in \left[n_b\right]$ and obtain $J_e(\theta_i)\in\J_{e}(\theta_i)$\\
    Compute $J_\varphi(\theta_i) = J_\mathcal{S}(\theta_i)^\top c(\theta_i) + \mathcal{S}(\theta_i)^\top J_c(\theta_i) + 2\lambda_\mathrm{p} \max \{0,e(\theta_i)\} J_e (\theta_i)$\\
    Update $\theta_{i+1} = \theta_i - \alpha_i J_\varphi(\theta_i)$\\
    Update the conic program parameters $A(\theta_{i+1}), b(\theta_{i+1}), c(\theta_{i+1})$\\
    }
    \caption{Loss-aware Distributional Robust Optimization}
    \label{alg:learning_ambiguity}
\end{algorithm}

\subsection{Convergence}\label{subsection:path_diff}
The gradient descent procedure in Algorithm~\ref{alg:learning_ambiguity} is guaranteed to converge to a critical point under mild conditions.

\begin{theorem}
Under \cref{ass:path_diff_bilevel}, if the step sizes $\alpha_i\geq 0 $ are square summable but not summable, and $\sup_i \theta_i < +\infty$, then $\theta_\infty:=\lim_{i\to \infty} \theta_i\in \operatorname{crit}\varphi$, where $\operatorname{crit}\varphi = \{ \theta: 0\in \J_{\varphi}(\theta \}$. If $\alpha_i\equiv \bar{\alpha}$, then for every $\epsilon>0$ there exists an $\bar{\alpha}>0$ such that $\limsup_{i\to\infty} \operatorname{dist}(\theta_i,\operatorname{crit}\varphi)\leq\epsilon$.
\end{theorem}
\begin{proof}
Definability of $e(\theta)$ follows immediately from the definability of $d(\mathbb{P}_1,\mathbb{P}_2;\theta)$, since all functions involved are definable, and inversion preserves definability \cite[Remark 2]{kurdyka1998gradients}. The result then follows from \cite[Theorem 3.2]{davis2020stochastic} for the vanishing step size case, and from \cite[Theorem 2]{bolte2024inexact} for the constant step size case.
\end{proof}
As a by-product of our problem formulation, we can also embed an OT-DRO problem as a general-purpose differentiable layer. Our analysis is general and extends beyond OT-based ambiguity set to \textit{any} ambiguity set \rz{that can be represented as a conic set}. This class is broad and encompasses many \rz{relevant} uncertainty descriptions \cite{ben2001lectures} including support-, moment- and entropic-based ambiguity sets. The differentiable layer can be used in any machine learning framework (e.g., PyTorch, TensorFlow, JAX) to learn the ambiguity set parameters and can be integrated into a larger layered architecture.

%% file: Contents/Numerics.tex
\subsection{Portfolio optimization}
\label{numerics:portfolioopt}
We consider a portfolio optimization problem where the goal is to find the optimal allocation of $k$ financial assets to maximize the profit from the investment. Mathematically,
$
    \ell(w,\xi) = -w^\top \xi,
$
where  $w \in \mathbb{R}^k$ is the vector of allocation weights, $\xi \in \mathbb{R}^k$ are the returns, and $\mathcal{W} = \left\{ w \in \mathbb{R}^k \; \vert\; \mathbf{1}^\top w =1,\; w \succeq 0\right\}$. 
\rz{We assume that the returns $\xi$ are distributed according to an unknown Gaussian probability distribution $\mathbb{P}$.}
\mf{The decision maker considers a nominal distribution $\hat{\mathbb{P}}\triangleq \mathcal{N}(\hat{\mu}, \hat{\Sigma})$ and robustifies against the mismatch via a parametric Wasserstein ambiguity set with $\kappa(\xi_1, \xi_2; \theta) = \|L^\top (\xi_1 - \xi_2)\|_2^2$ with $\theta = L \in \mathbb{L}_{++}^k$ being the tuning parameter. Formally, the transportation cost corresponds to the Mahalanobis distance with weight $L$.} 
To hedge against the tail-risk, we consider the Conditional Value at Risk (CVaR) of the loss $\ell$ \cite{rockafellar2000optimization}. The OT-DRO problem then reads
\begin{equation}
    \min_{w\in\mathcal{W}, L \in \mathbb{L}_{++}^k} \max_{\mathbb{Q} \in \mathcal{B}_\varepsilon(\hat{\mathbb{P}};\theta)} \text{CVaR}_\gamma^\mathbb{Q}\left( -w^\top\xi \right), \label{eq:opt_dro_portfolioopt}
\end{equation}
where
$$\text{CVaR}_\gamma^\mathbb{Q}(X) = \min_{\tau \in \mathbb{R}} \left\{ \tau + \frac{1}{\gamma} \mathbb{E}_{X \sim \mathbb{Q}}[\max(0, X - \tau)]\right\}$$
\mf{and $\mathcal{B}_\varepsilon(\hat{\mathbb{P}};\theta) = \{ \mathbb{Q} \in \mathcal{P}_g(\mathbb{R}^k) \; | \; d(\mathbb{Q}, \hat{\mathbb{P}}; \theta) \leq \varepsilon \}$ is a parametrized OT ambiguity set restricted to contain only Gaussians.\footnote{\mfnew{This restriction is solely done to simplify the implementation of the bootstrapping procedure, as it suffices to bootstrap the first two moments of the nominal distribution to fully characterize the boostrapped ones. Appendix~\ref{app:furtherresults:portfolioopt} provides a fully data-driven implementation that does not rely on such restriction.}}}
We defer the reformulation of \eqref{eq:opt_dro_portfolioopt} to Appendix~\ref{app:dro_cp_reformulation:portfolioopt:gaussianref} and the experimental details to Appendix~\ref{app:numerics_details:portfolioopt}. Further, we report results on portfolio optimization using a discrete reference distribution in Appendix~\ref{app:furtherresults:portfolioopt}.

We first exemplify how our procedure works on a specific problem instance with $k = 2, J=30, n_b = 20, \gamma = 0.05, \beta = 0.1$.
%
The results are shown in Figure~\ref{fig:single_example_truedist_L_decisions}.
\rz{As the iterations progress, the $L$ matrix changes, increasing the weight of the probability mass in the lower-left corner corresponding to adversary returns. This effectively decreases the worst-case $\mathrm{CVaR}_\gamma^{\mathbb{Q}^\star}$ as distributions assigning high probability mass in this direction are excluded from the ambiguity set.}
\rz{At the same time, this leads to an improved} out-of-sample performance $\text{CVaR}_\gamma^{\mathbb{P}}$, indicating a reduction in the conservatism of the solution (center).
Meanwhile, the true distribution is still contained in the ambiguity set with high probability via \eqref{eq:constraint:boot}, ensuring the required robustness properties \mf{(see later discussion on Figure~\ref{fig:effectiveness_bootstrapping_gaussian_cvar})}.
\mf{The gap between the blue and the red curves in Figure~\ref{fig:single_example_truedist_L_decisions} (center) reflects the ``price of robustness'' due to only knowing the distribution through samples, as $\mathcal{B}_\varepsilon(\hat{\mathbb{P}};\theta)$ contains also different, and possibly more adversarial, distributions.}
%
%

We validate our procedure on $50$ independent experiments with $k=3$, each using a different true Gaussian distribution $\mathbb{P}$. For each distribution, $10$ distinct datasets are sampled. To evaluate the effectiveness of our procedure, we monitor the relative improvement of the worst-case objective 
\begin{equation}
    f_0 = \mathrm{CVaR}_{\gamma}^{\xi\sim\mathbb{Q}^\star(L_0)}\left[ -(w^\star(L_0))^\top \xi \right] \text{ and } f^\star = \mathrm{CVaR}_{\gamma}^{\xi\sim\mathbb{Q}^\star(L^\star)}\left[ -(w^\star(L^\star))^\top \xi \right] \nonumber
\end{equation}
and of the out-of-sample performance
\begin{equation}
    \ell_0 = \mathrm{CVaR}_{\gamma}^{\xi\sim\mathbb{P}}\left[ -(w^\star(L_0))^\top \xi \right] \text{ and } \ell^\star = \mathrm{CVaR}_{\gamma}^{\xi\sim\mathbb{P}}\left[ -(w^\star(L^\star))^\top \xi \right]. \nonumber
\end{equation}
Figure~\ref{fig:single_example_truedist_L_decisions} (right) confirms that a reduction of conservatism is observed on average across all problem instances, with the relative improvement being larger for a smaller number of samples. 
\begin{figure}[htbp]
    \centering
    {\includegraphics[width=\textwidth]{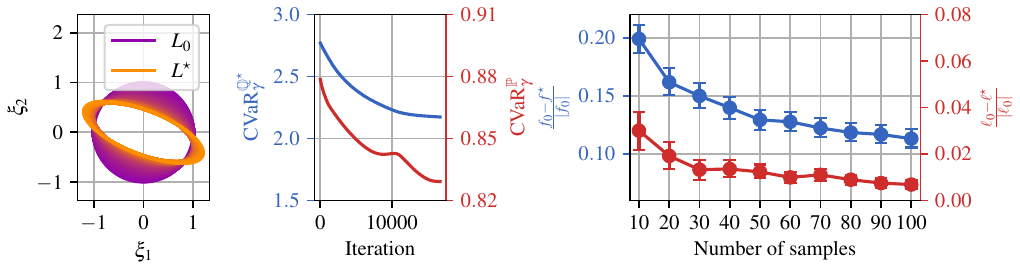}}
    \caption{Results of the bilevel optimization. Change in transportation cost parametrization (left), convergence plot over iterations (center), average improvement over multiple experiments w.r.t. number of samples (right)}
    \label{fig:single_example_truedist_L_decisions}
\end{figure}
%

\mf{We assess the coverage of the true distribution by tracking the parametrized distance $d(\mathbb{Q},\hat{\mathbb{P}}; \theta)$ across all experiments in Figure~\ref{fig:effectiveness_bootstrapping_gaussian_cvar}. The results confirm that our procedure reliably contains the true distribution within the ambiguity set with high probability, preserving the out-of-sample guarantees of OT-DRO. In contrast, omitting the coverage constraint in \eqref{eq:constraint:boot} may lead to excessive shrinkage of the ambiguity set, eventually excluding the true distribution.}

\begin{figure}[htbp]
    \centering
    \includegraphics[width=0.9\linewidth]{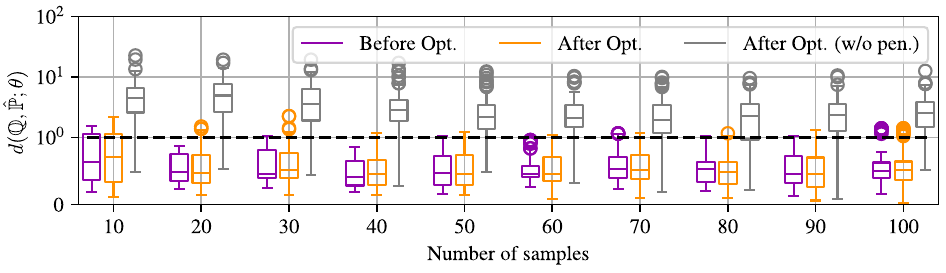}
    \caption{{Normalized distance $d(\mathbb{Q}, \hat{\mathbb{P}}; \theta)$ between reference and true distribution before (violet) and after (orange) the bilevel procedure is applied (i.e., for $L_0$ and $L^\star$, respectively). The grey boxplot corresponds to running our scheme without enforcing \eqref{eq:constraint:boot}, that is, by setting $\lambda_p = 0$.}}
    \label{fig:effectiveness_bootstrapping_gaussian_cvar}
\end{figure}

\subsection{Linear regression}
\label{numerics:linear:regression}
Next, we consider the distributionally robust linear regression task \cite{kuhn2024wassersteindistributionallyrobustoptimization,shafieezadeh-abadeh_regularization_2019,blanchet_robust_2019,chen2021distributionally}, where the goal is learning a linear regressor that performs well under uncertainty in the data distribution. Let $x\in\R^k$ and $y\in\R$ be the independent and dependent variables, respectively, and let $\xi = (x,y) \in \mathbb{R}^{k+1}$. We parametrize the transportation cost as $\kappa(\xi_1,\xi_2;\theta) = \norm{L^\top (\xi_1-\xi_2)}_2$ with $\theta=L\in\mathbb{L}_{++}^k$, and consider the loss function
$
    \ell_1(w,\xi) = \abs{(-w,1)^\top \xi} = \abs{\bar{w}^\top \xi},
$
where the decision $w\in\mathbb{R}^k$ represents the weights of the linear model. The distributionally robust linear regression problem reads
\begin{equation}\label{eq:dro:linear:regression}
    \min_{w\in\mathbb{R}^k} \max_{\mathbb{Q} \in \mathcal{B}_\varepsilon(\hat{\mathbb{P}};\theta)} \mathbb{E}_{\xi\sim\mathbb{Q}}\left(  \ell_1(w,\xi)\right), 
\end{equation}%
where $\hat{\mathbb{P}} = \frac{1}{J} \sum_{j=1}^J \delta_{\hat{\xi}_j}$ with $\hat{\xi}_j = \{(\hat{x}_j, \hat{y}_j)\}$.  
We defer the reformulation of \eqref{eq:dro:linear:regression} to Appendix~\ref{app:dro_cp_reformulation:linreg:1type} and the experimental details to Appendix~\ref{app:numerics_details:linreg}. Further, we report the analogous case with a squared loss function $\ell_2(w,\xi) = \left((-w,1)^\top \xi\right)^2$ in Appendix~\ref{app:furtherresults:linreg:squared}.

As before, we begin by showing the results of a single problem instance, where $\xi=(x,y)\in\R^2$ is generated with the following linear model corrupted by white noise
\begin{equation}
    y = w x + \mathrm{e}, \text{ where } \mathrm{e} \sim \mathcal{N}(0,\sigma) \text{ and } x\sim \mathcal{U}(-10.0,10.0). \label{eq:linreg_datagen_model}
\end{equation}
The weight $w$ is deterministic and set to $1$, and the standard deviation of the noise is $\sigma=10$.

Figure~\ref{fig:linreg_single_example} shows the true distribution of $\xi$ and the $J=20$ samples drawn from it (left), the unit transportation cost ellipses across iterations (center), and the effect of our procedure on the linear model (right). 
\begin{figure}[htbp]
    \centering
    \includegraphics[width=\textwidth]{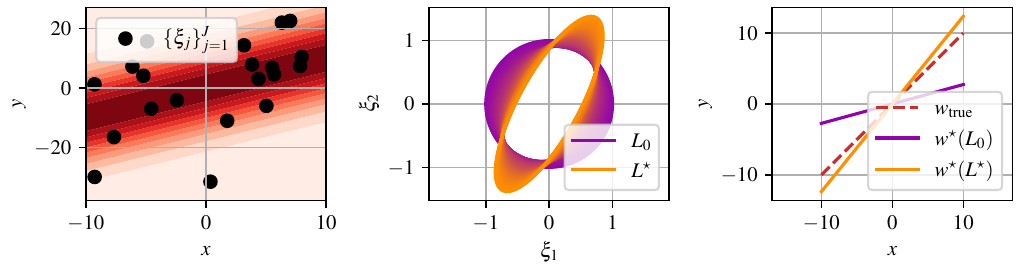}
    \caption{Underlying true distribution $\mathbb{P}$ (red contour) and samples (black), unit-cost ellipses defined by the matrices $L_i$ in the upper-level optimization, and optimal decisions before ($w^\star(L_0)$) and after ($w^\star(L^\star)$) optimization. We set $n_b=20$.}
    \label{fig:linreg_single_example}
\end{figure}

Figure~\ref{fig:linreg_errors} (left) shows the worst-case expected absolute error $\mathrm{e}_{\mathrm{wc}}(L) = \mathbb{E}_{\xi\sim\mathbb{Q}^\star(L)}\left( \ell_1(w^\star(L),\xi) \right)$, corresponding to the upper level objective, and the expected absolute error on the true data generating process $\mathrm{e}_{\mathrm{oos}}(L) = \mathbb{E}_{\xi\sim\mathbb{P}}\left( \ell_1(w^\star(L),\xi) \right)$, approximated with $10^7$ independent samples. Both metrics decrease across iterations, indicating a reduction in conservatism of the decision.
We corroborate our results on $10$ independent experiments, each using a different true distribution $\mathbb{P}$ from which $10$ datasets are generated. To evaluate the proposed method, we consider the same relative improvement metrics used before, adapted to the linear regression loss function. Results, shown in Figure~\ref{fig:linreg_errors} (right), suggest once again that the proposed method reduces, on average, the conservatism of the decision.
  
\begin{figure}[htbp]
    \centering
    \includegraphics[width=\linewidth]{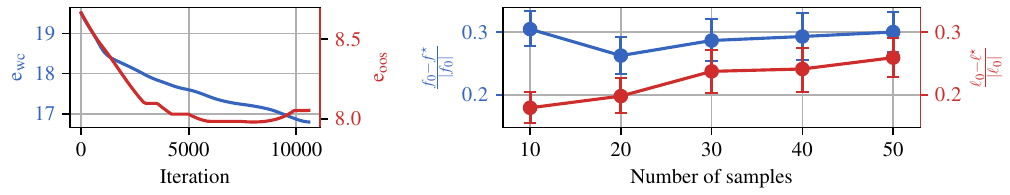}
    \caption{Expected absolute error on the worst-case distribution $\mathrm{e}_\mathrm{wc}$ and expected absolute error on the true data generating process $\mathrm{e}_\mathrm{oos}$ approximated with $10$ million samples (left). Average improvement over multiple experiments w.r.t. the number of samples (right).}
    \label{fig:linreg_errors}
\end{figure}

%% file: Contents/Conclusions.tex
\mfnew{

\textbf{Limitations.} A key limitation of our work lies in the absence of theoretical guarantees for the coverage constraint (Eq.~\ref{constraint:upper}), which we currently enforce heuristically through a bootstrapping procedure. Although our experiments highlight the practical effectiveness and robustness of the proposed algorithm, providing a rigorous lower bound on the probability of covering the true distribution under resampling-based methods remains an open challenge. Addressing this requires a deeper understanding of the interplay between $J$ and $n_b$, as preliminarily discussed in \cite{summers2021distributionally}, and represents an important direction for future research.

\textbf{Outlook.} In view of the growing popularity of the OT-DRO model, the results in this work carry important practical implications: we expect our end-to-end pipeline to result in less conservative, yet reliable decisions across several domains, including finance and machine learning, as demonstrated in Section \ref{ch:numerics}.
} 

%% file: Contents/Appendices.tex
\newpage
\begin{appendices}

\listofappendices

\section{Finite-dimensional DRO reformulation as Conic Program}
\label{app:dro_cp_reformulation}
\input{Contents/Appendix_A_DRO_reformulation}

\section{Functions definable in an o-minimal structure}
\label{app:definability}
\input{Contents/Appendix_A2_path_diff.tex}

\section{Differentiating through Conic Programs}
\label{app:diff_through_conic_programs}
\input{Contents/Appendix_B_differentiating_conic_program}

\section{Differentiating through Optimal Transport Problems}
\label{app:diff_ot}
\input{Contents/Appendix_C_differentiate_Wassertein.tex}

\section{Numerical Experiments Details}
\label{app:numerics_details}
\input{Contents/Appendix_D_numerical_details}

\section{Further Experimental Results}
\label{app:further_results}
\input{Contents/Appendix_E_additional_results}

\section{Computational Complexity and Sensitivity Analysis of Hyperparameters}
\label{app:complexity_and_sensitivity}
\input{Contents/Complexity_and_Sensitivity.tex}

\end{appendices}

%% file: Contents/Appendix_A_DRO_reformulation.tex
\subsection{General reformulation}
\label{app:dro_cp_reformulation:primal}
We provide a finite-dimensional convex reformulation for a general instance of the OT-DRO problem \eqref{eq:decision_making_under_uncertainty} (and of its dual) under mild regularity conditions. 
\subsubsection{Primal problem}
Consider the primal OT-DRO problem 
\begin{equation}
    \inf_{w\in\mathcal{W}} \; \sup_{\mathbb{Q} \in \mathcal{B}_\varepsilon(\hat{\mathbb{P}} )}\mathop{\mathbb{E}}_{\xi \sim \mathbb{Q}} \left[\ell(w,\xi) \right].
    \label{eq:primal_dro}
\end{equation}
with empirical nominal distribution $\hat{\mathbb{P}} = \frac{1}{J}\sum_{j=1}^J \delta_{\hat{\xi}_j}$, and OT ambiguity set defined as
\begin{equation}
    \mathcal{B}_\varepsilon(\hat{\mathbb{P}}) = \left\{\mathbb{Q} \in \mathcal{P}(\Xi) \; \vert\; d(\mathbb{Q}, \hat{\mathbb{P}}) \leq \varepsilon\right\},
    \label{eq:ambiguity_set:appendix}
\end{equation}
where $d(\mathbb{Q}, \hat{\mathbb{P}}): \mathcal{P}(\Xi) \times \mathcal{P}(\Xi) \rightarrow [0,+\infty)$ is defined as
\begin{equation*}
    d(\mathbb{Q}, \hat{\mathbb{P}}) = \inf_{\pi \in \Pi(\mathbb{Q}, \hat{\mathbb{P}})} \int_{\Xi \times \Xi} \kappa(\xi, \hat{\xi}) \; \mathrm{d}\pi(\xi, \hat{\xi}),
\end{equation*}
for a certain transportation cost $\kappa(\xi_1, \xi_2):\Xi \times \Xi \rightarrow [0, +\infty)$. To provide a convex finite-dimensional reformulation of \eqref{eq:primal_dro}, some regularity conditions need to be satisfied \cite{shafieezadehabadeh2023newperspectivesregularizationcomputation}.

\begin{assumption}[Transportation cost] The transportation cost has to satisfy the properties:
    \begin{enumerate}
        \item [i] $\kappa(\xi_1, \xi_2) \geq 0$ always, and $\kappa(\xi_1, \xi_2) = 0 \iff \xi_1 = \xi_2$.
    \item [ii.] $\kappa(\xi_1,\xi_2)$ is lower semicontinuous in $(\xi_1,\xi_2)$ and convex in its first argument.
    \item [iii.] There exists a reference point $\hat{\xi} \in \mathbb{R}^d$ such that $\mathbb{E}_{\xi \sim \hat{\mathbb{P}}}[\kappa(\xi, \hat{\xi})] < +\infty$.
    \item [iv.] There exists a metric $d(\xi_1, \xi_2)$ on $\mathbb{R}^d$ with compact sublevel sets such that $\kappa(\xi_1, \xi_2) \geq d^p(\xi_1, \xi_2)$ for some $p \in \mathbb{N}$.
\end{enumerate}
    \label{ass:transportation_cost}
\end{assumption}
Specifically, Assumptions \ref{ass:transportation_cost}(iii-iv) allows to prove that the ambiguity set $ \mathcal{B}_\varepsilon(\hat{\mathbb{P}})$ is weakly compact. Assumptions \ref{ass:transportation_cost}(i-ii) automatically implies that $\kappa(\xi_1, \xi_2)$ is proper, convex and lower semicontinuous in $\xi_1$ for any fixed $\xi_2$ and that it is proper, convex and lower semicontinuous in $\xi_2$ for any fixed $\xi_1$.

Next, we assume the following properties for the loss function.
\begin{assumption}[Loss function] The loss function is representable as point-wise maximum of finitely many functions
    \begin{equation*}
        \ell(w, \xi) = \max_{i\in I}\; \ell_i(w, \xi),
    \end{equation*}
    where each $\ell_i$ is proper, convex, and lower semicontinuous in $w,$ and $-\ell_i$ are proper, convex, and lower semicontinuous in $\xi$. 
    \label{ass:loss_function}
\end{assumption}
 Moreover, we enforce the following for the support set.
\begin{assumption}[Support of the uncertainty] The support set is representable as
    \begin{equation*}
        \Xi = \left\{\xi \in \mathbb{R}^d \; \vert\; f_r(\xi) \leq 0 \;\forall r \in [R] \right\},
    \end{equation*}
    where each $f_r$ is proper, convex, and lower semicontinuous.
    \label{ass:support_uncertainty}
\end{assumption}

The finite-dimensional convex reformulation of \eqref{eq:primal_dro} relies on the following technical condition.
\begin{assumption}[Slater conditions] It holds:
    \begin{enumerate}
    \item [i.] For every $j \in [J], \hat{\xi}_j \in \mathrm{relint}(\mathrm{dom}(\kappa(\cdot, \hat{\xi}_j)))$ is a Slater point for the support set $\Xi$.
    \item [ii.] The feasible set $\mathcal{W}$ admits a Slater point.
    \end{enumerate}
    \label{ass:slaters_cond_support}
\end{assumption}

Let Assumptions \ref{ass:transportation_cost}, \ref{ass:loss_function}, \ref{ass:support_uncertainty}, and \ref{ass:slaters_cond_support} hold. Further, let $\varepsilon > 0$ and assume $\mathcal{W}$ is compact and convex. Then, the primal OT-DRO problem \eqref{eq:primal_dro} has the same infimum of the following finite-dimensional convex optimization problem \cite[Proposition~2.13]{shafieezadehabadeh2023newperspectivesregularizationcomputation}: 
\begin{equation}
    \begin{aligned}
        \inf_{} \quad & \lambda \varepsilon + \frac{1}{J}\sum_{j \in \left[J \right]} s_j \\
            \textrm{subject to} \quad & w \in \mathcal{W}, \lambda \in \mathbb{R}_+ , \tau_{ijr} \in \mathbb{R}_+, s_j \in \mathbb{R}, a_{ij}^\ell, a_{ij}^c, a_{ijr}^f \in \mathbb{R}^d \quad \forall i \in \left[ I \right],\;j \in \left[ J \right],\;r \in \left[ R \right] \\
            & (-\ell_i)^{*2}(w, a_{ij}^l) \; + \; \lambda \kappa^{*1}\left(\frac{a_{ij}^c}{\lambda}, \hat{\xi}_j\right) \; + \; \sum_{r\in\left[ R \right]} \tau_{ijr} f_r^*\left(\frac{a_{ijr}^f}{\tau_{ijr}}\right) \; \leq \; s_j \quad \forall i \in \left[ I \right],\;j \in \left[ J \right]\\
            & a_{ij}^\ell + a_{ij}^c + \sum_{r \in \left[R\right]} a_{ijr}^f = 0 \quad \forall i \in \left[ I \right],\;j \in \left[ J \right].
    \end{aligned}
    \label{eq:cp_reform_primal_dro}
\end{equation}    
Here, for any function of two arguments $f(\cdot,\cdot)$, we use $f^{*1}$ and $f^{*2}$ to denote the convex conjugate of the function with respect to its first and second argument while keeping the other argument fixed. A superscript $f^*$ on a function in a single argument denotes the convex conjugate.

\subsubsection{Dual problem}
\label{app:dro_cp_reformulation:dual}
Similarly, the dual OT-DRO problem 
\begin{equation}
    \sup_{\mathbb{Q} \in \mathcal{B}_\varepsilon(\hat{\mathbb{P}})} \inf_{w\in\mathcal{W}}  \; \mathop{\mathbb{E}}_{\xi \sim \mathbb{Q}} \left[\ell(w,\xi) \right]
    \label{eq:dual_dro}
\end{equation}
admits a tractable reformulation under certain regularity conditions. Specifically, in addition to Assumptions \ref{ass:transportation_cost}, \ref{ass:loss_function}, \ref{ass:support_uncertainty}, and \ref{ass:slaters_cond_support}, we require the following.
\begin{assumption}[Dual regularity conditions] One of the following three conditions has to be satisfied: (i) $\mathbb{E}_{\xi \sim \mathbb{Q}}[\ell(w,\xi)]$ is inf-compact in $w \in \mathcal{W}$ for some $\mathbb{Q} \in \mathcal{B}_\varepsilon(\hat{\mathbb{P}})$, (ii) $\Xi$ is compact, or (iii) $\kappa(\cdot, \hat{\xi})$ grows superlinearly in its first argument.
    \label{ass:compactness_superlinearity}
\end{assumption}

\begin{assumption}[Feasible decision set]\label{ass:feasibile:set}
The feasible set is representable as
\begin{equation*}
\mathcal{W} = \{ w \in \mathbb{R}^k \: | \: g_l(w) \leq 0 \: \forall l \in [L]\},
\end{equation*}
where each $g_l$ is proper, convex, and lower semicontinuous.
\end{assumption}

Let all the above assumptions hold and let $\varepsilon > 0$. Then, the dual DRO problem \eqref{eq:primal_dro} has the same supremum of following finite-dimensional convex optimization problem \cite[Proposition~2.15]{shafieezadehabadeh2023newperspectivesregularizationcomputation}: 
\begin{equation}
    \begin{aligned}
        \max_{} \quad & - \sum_{i \in \left[I \right]}\sum_{j \in \left[J \right]} q_{ij} \ell_i^{*1}\left(\alpha_{ij}/q_{ij},\; \hat{\xi}_{j} + b_{ij}/q_{ij}\right) - \sum_{l \in \left[L \right]} \nu_l g_l^*(\beta_l / \nu_l) \\
        \textrm{subject to}& \quad q_{ij},\; \nu_l \in \mathbb{R}_+, b_{ij}, \alpha_{ij}, \beta_l \in \mathbb{R}^k \quad \forall i \in [I],j \in [J], l \in [L] \\
        & \quad f_r\left( \hat{\xi}_j + b_{ij} / q_{ij} \right) \leq 0 \quad \forall i \in [I], j \in [J], r \in [R] \\
        & \quad \sum_{i\in\left[I\right]}q_{ij} = p_j \quad \forall j \in [J]\\
        & \quad \sum_{i \in \left[ I \right]}\sum_{j \in \left[ J \right]}\alpha_{ij} + \sum_{l\in\left[L\right]}\beta_l = 0 \\
        & \quad \sum_{i \in \left[ I \right]} \sum_{j \in \left[ J \right]} q_{ij}\kappa(\hat{\xi}_j + b_{ij}/q_{ij}, \hat{\xi}_j) \leq \varepsilon.
    \end{aligned}
    \label{eq:cp_reform_dro_dual}
\end{equation}

The optimal variables of this problem $\left\{q^\star_{ij},\; b^\star_{ij},\; \alpha^\star_{ij} \right\}_{i,j}$ and $\left\{\nu^\star_{l},\; \beta^\star_{l} \right\}_{l}$ can then be used to construct $\mathbb{Q}^\star$ as follows. Consider $\mathcal{I}_j^+=\{i \in [I]\;\vert\;q_{ij}^\star > 0\}$, $\mathcal{I}_j^0=\{i \in [I] \;\vert\;q_{ij}^\star = 0,\; b_{ij}^\star = 0\}$, and $\mathcal{I}_j^\infty=\{i\in [I]\;\vert\;q_{ij}^\star = 0,\; b_{ij}^\star \neq 0\}$. Then
\begin{equation*}
    \mathbb{Q}^\star = \sum_{j\in [J]} \sum_{i \in \mathcal{I}_j^+} q_{ij}^\star \delta_{\hat{\xi}_j + b_{ij}^\star/q^\star_{ij}} \text{ if } \mathcal{I}_j^\infty = \emptyset \; \forall j \in [J]
\end{equation*}

If there is a $j$ for which $\mathcal{I}_j^\infty$ is not empty, i.e. there exist $i,j$ such that $q_{ij}^\star=0$ but $b_{ij}^\star \neq 0$, then only an asymptotic sequence of probability distributions can be generated that converges to the optimal solution of the dual DRO problem \eqref{eq:dual_dro}. For $n \rightarrow \infty$, this construction reads \cite{shafieezadehabadeh2023newperspectivesregularizationcomputation}:
\begin{equation*}
    \mathbb{Q}^\star \mathop{\leftarrow}_{n\rightarrow \infty} \mathbb{Q}(n) = \sum_{j\in [J]} \sum_{i \in \mathcal{I}_j^+ \cup \; \mathcal{I}_j^\infty} \begin{cases}
        q_{ij}^\star (1-\abs{\mathcal{I}_j^\infty}/n) \; \delta_{\hat{\xi}_j + b_{ij}^\star/q^\star_{ij}} &\quad \text{ if } i \in \mathcal{I}_j^+\\
        \frac{p_j}{n} \; \delta_{\hat{\xi}_j + n\,b_{ij}^\star/p_{j}} &\quad \text{ if } i \in \mathcal{I}_j^\infty
    \end{cases}  
\end{equation*}

\subsection{Reformulation of the Portfolio Optimization problem}
\label{app:dro_cp_reformulation:portfolioopt}
This section provides convex reformulations for the distributionally robust portfolio optimization task with parametrized transportation costs under different problem settings.

\subsubsection{Discrete Reference Distribution}
\label{ch:portfolioopt_primal_empref_reformulation}
In case of discrete nominal distributions, we can directly invoke the results from Appendix~\ref{app:dro_cp_reformulation:primal}. The conjugate function of the objective in the case of a bilinear loss function, as in the portfolio optimization example, i.e.,
\begin{equation*}
    -\ell(x,y) = x^\top y    
\end{equation*}
evaluates as
\begin{equation}
    (-\ell)^{*2}(x, y) = \sup_{\zeta \in \mathbb{R}^k} \left\{ \zeta^\top (y - x) \right\} = \begin{cases}
        0 &\text{if } y = x \\
        \infty &\text{otherwise}
    \end{cases}.\label{eq:conj_loss_po}
\end{equation}

\paragraph{type-1 Mahalanobis distance.}
\label{app:dro_cp_reformulation:portfolioopt:1type}We consider a parametrized transportation cost of the form $\kappa(\xi_1, \xi_2; \theta) = \|L^\top(\xi_1 - \xi_2)\|_2$ with $\theta = L \in \mathbb{L}_{++}^k$, and let $\Xi = \mathbb{R}^d$ for $d = k$.       
Then, the conjugate function of 
\begin{equation*}
    \kappa(y,a; \theta) = \norm{L^\top (y-a)}_2,  
\end{equation*}
is given by
\begin{align}
    \kappa^{*1}(y,a) &= \sup_{x\in\mathbb{R}^d}\left\{ x^\top y - c(x,a) \right\} \\
    &= \sup_{x\in\mathbb{R}^d}\left\{ x^\top y - \norm{L^\top \left(x - a\right)}_2 \right\} \\
    &= a^\top y + \sup_{z\in\mathbb{R}^d} \left\{ z^\top y - \norm{L^\top z}_2 \right\} \\
    &= a^\top y + \sup_{w\in\mathbb{R}^d} \left\{ y^\top L^{-\top}w - \norm{w}_2 \right\} \label{eq:convconjtransportationcost_4}\\
    &= \begin{cases}
        a^\top y &\text{if } \norm{L^{-1}y}_2 \leq 1 \\
        \infty & \text{otherwise}
    \end{cases}\label{eq:convconjtransportationcost_5}.
\end{align}
where the first and second equality follow from the definition of convex conjugate and from the parametrized transportation cost, respectively; the third one is obtained by introducing the change of variables $z=x-a$; and the fourth one by letting $w=L^\top z$. As for the last equality, we proceed as follows. Let $w = \alpha \cdot \left( L^{-1}y \right) + \beta \cdot \left(u\right)$, where the vector $u$ is orthogonal to $L^{-1}y$ and $u$ has length one, the objective function inside the supremum of \eqref{eq:convconjtransportationcost_4} follows as
\begin{equation*}
    y^\top L^{-\top}w - \norm{w}_2  = \alpha y^\top L^{-\top} L^{-1} y + \beta \left(u\right)^\top L^{-1}y   - \norm{\alpha L^{-1}y + \beta u}_2 = f_\mathrm{sup}(\alpha, \beta).
\end{equation*}
However, as $u$ is orthogonal to $L^{-1}y$ and thus $\left(u\right)^\top L^{-1}y = 0$, we can simply choose $\beta^\star = 0$, i.e., $w^\star$ has to be aligned with $L^{-1}y$. At this point, we turn our attention to $\alpha$:        
\begin{align*}
    f_\mathrm{sup}(\alpha) &= \alpha \norm{P^{-\top}y}_2^2 - \abs{\alpha} \norm{P^{-\top}y}_2 \\
    &= \norm{P^{-\top}y}_2 \left( \alpha \norm{P^{-\top}y}_2 - \abs{\alpha}\right)
\end{align*}
We distinguish two cases:
\begin{itemize}
    \item $\alpha \geq 0$: $f_\mathrm{sup}(\alpha) = \alpha \norm{P^{-\top}y}_2 \left( \norm{P^{-\top}y}_2 - 1\right)$ \\
    The objective function $f_\mathrm{sup}(\alpha)$ in this region is unbounded above if $\norm{P^{-\top}y}_2 - 1 > 0$.\\
    If $\norm{P^{-\top}y}_2 - 1 > 0$, then the supremum is zero (as then $f_\mathrm{sup}(\alpha) \leq 0 \; \forall \alpha \geq 0$.
    \item $\alpha < 0$: $f_\mathrm{sup}(\alpha) = \alpha \norm{P^{-\top}y}_2 \left( \norm{P^{-\top}y}_2 + 1\right)$ \\
    The objective function is a linear function in $\alpha$ with positive slope. Its value is always negative (as $\alpha < 0$), which means that the optimal value of $\alpha$ is $\alpha^\star = 0$, and thus $f_\mathrm{sup}^\star = 0$).
\end{itemize}
Combining both results directly leads to \eqref{eq:convconjtransportationcost_5}. The resulting reformulation for the distributionally robust portfolio optimization problem in this setting reads
\begin{equation*}
    \begin{aligned}
        \inf_{} \quad & \lambda \varepsilon + \frac{1}{J}\sum_{j \in \left[J \right]} s_j \\
        \textrm{subject to} \quad & w \in \mathcal{W}, \lambda \in \mathbb{R}_+, s_j \in \mathbb{R} \quad \forall j \in \left[ J \right]\\
        & -w^\top \hat{\xi}_j \; \leq \; s_j \quad \forall j \in \left[ J \right]\\
        & \norm{L^{-1} w}_2 \leq \lambda.
    \end{aligned}
\end{equation*}
Next, we introduce the auxiliary variable $u$ such that $Lu = w \implies u = L^{-1}w$. Further, we introduce the variables $z_j = s_j + w^\top \hat{\xi}_j \geq 0$ and recognize that $z_j = 0 \; \forall j$ at optimality. This leads to
\begin{equation}\label{eq:DRO:port:step}
    \begin{aligned}
        \min_{\lambda, w, u} \quad & \lambda \varepsilon - w^\top \frac{1}{J}\sum_{j \in \left[J \right]}\hat{\xi}_j \\
        \textrm{subject to} \quad & w \in \mathbb{R}_+^k, \lambda \in \mathbb{R}_+ \\
        & \mathbf{1}^\top w = 1\\
        & Lu = w \\
        & \norm{u}_2 \leq \lambda.
    \end{aligned}
\end{equation}
In turn, \eqref{eq:DRO:port:step} can be cast as conic program in standard form \eqref{standard:conic} by using the following definitions:
\begin{align}
    s &= (s_1, s_2, s_3, s_4) \in \mathcal{K} := \left\{0\right\} \times \left\{0\right\}^{k} \times \mathbb{R}_+^{k} \times \mathrm{SOC}(k+1) \\
    x &= (w, \lambda, u) \in \mathbb{R}^{k+1+k}\\
    c &= (-\frac{1}{J}\sum_{j \in \left[J \right]}\hat{\xi}_i,\, \varepsilon,\, 0,\, ...,\, 0)\\
    b &= (1,\, 0,\, ...,\, 0) \\
    A &= \begin{bmatrix}
        \mathbf{1}_k^\top & 0 & 0 \\
        -I_k & 0 & L \\
        -I_k & 0 & 0 \\
        0 & -1 & 0\\
        0 & 0 & -I_k
    \end{bmatrix}. \label{eq:empirical_primal_dro_reform_A_matrix}
\end{align}

\paragraph{2-type Mahalanobis distance.}
\label{app:dro_cp_reformulation:portfolioopt:2type}We consider a different parametrization for the transportation cost given by $\kappa(x,y; \theta) = \norm{L^\top (x-y)}_2^2$, which admits the conjugate 
\begin{align*}
    \kappa^{*1}(y,a) &= \sup_{x\in\mathbb{R}^d}\left\{ x^\top y - \norm{L^\top \left(x - a\right)}_2^2 \right\} \\
    &= a^\top y + \frac{1}{4}y^\top (L L^\top)^{-1} y.
\end{align*}
Again, we invoke the results of Appendix~\ref{app:dro_cp_reformulation:primal}, letting $\Xi = \mathbb{R}^d$ and using the conjugate of the bilinear loss function from \eqref{eq:conj_loss_po}. We obtain the following convex problem:
\begin{equation*}
    \begin{aligned}
        \min \quad & \lambda \varepsilon^2 + \frac{1}{J}\sum_{j \in \left[J \right]} s_j \\
        \textrm{subject to} \quad & w \in \mathcal{W}, \lambda \in \mathbb{R}_+, s_j \in \mathbb{R} \quad \forall j \in \left[ J \right]\\
        & \frac{1}{4\lambda} w^\top L^{-\top}L^{-1} w - \hat{\xi}_j^\top w \leq s_j \quad \forall j \in \left[ J \right].
    \end{aligned}
\end{equation*}
Towards reformulating it as a conic program in standard form \eqref{standard:conic}, we introduce the variables $t_j = \hat{\xi}_j^\top w + s_j$ as well as the vector $z$ such that $w = Lz \implies z = L^{-1}w$, yielding
\begin{equation*}
    \begin{aligned}
        \min \quad & \lambda \varepsilon^2 + \frac{1}{J}\sum_{j \in \left[J \right]} s_j \\
        \textrm{subject to} \quad & w \in \mathcal{W}, \lambda \in \mathbb{R}_+, t_j \in \mathbb{R}, z \in \mathbb{R}^k, s_j \in \mathbb{R} \quad \forall j \in \left[ J \right]\\
         & t_j = s_j + \hat{\xi}_j^\top w \quad \forall j \in \left[ J \right]\\
        & w = Lz \\
        & z^\top z \leq 4\lambda t_j \quad \forall j \in \left[ J \right].
    \end{aligned}
\end{equation*}
Using the second-order cone representation \cite{ben2001lectures}:
\begin{equation*}
    z^\top z \leq t\cdot a, t\geq 0, a \geq 0 \quad \iff \quad \norm{(2z, t-a)}_2 \leq t + a, t\geq 0, a\geq 0,
\end{equation*}
we finally get
\begin{equation}\label{eq:dro:2:step}
    \begin{aligned}
        \min \quad & \lambda \varepsilon^2 + \frac{1}{J}\sum_{j \in \left[J\right]} (t_j - \hat{\xi}_j^\top w) \\
        \textrm{subject to} \quad & w \in \mathcal{W}, \lambda \in \mathbb{R}_+, u_j \in \mathbb{R}^{k+1}, t_j \in \mathbb{R}, v_j \in \mathbb{R} \quad \forall j\in \left[ J \right]\\
        & v_j = 4\lambda + t_j  \quad \forall j \in \left[ J \right]\\
        & \norm{u_j}_2 \leq v_j \quad \forall j \in \left[ J \right]\\
        & w = Lz \\
        & u_j = \begin{bmatrix}
            2 z \\
            4 \lambda - t_j
        \end{bmatrix} \quad \forall j \in \left[ J \right].
    \end{aligned}
\end{equation}
In turn, \eqref{eq:dro:2:step} can be readily written in standard form \eqref{standard:conic} by using the following definitions:
\begin{align*}
    s &= (s_1, \,s_2 ,\, v_1,\, u_1,\, ...,\, v_J,\, u_J) \in \mathcal{K} := \left\{0\right\}^{1+k} \times \mathbb{R}_+^{k+1} \times \{\mathrm{SOC}(k+2)\}^J \\
    x &= (t_1,\, ...,\, t_J, w, \lambda, z) \in \mathbb{R}^{J+k+1+k}\\
    c &= \textstyle \left(\frac{1}{J},\, ...\,,\frac{1}{J},\, -\frac{1}{J}\sum_{j \in \left[J \right]}\hat{\xi}_j,\, \varepsilon^2,\, 0,\, ...,\, 0 \right)\\
    b &= (1,\, 0,\, ...,\, 0) \\
    A &= \begin{bmatrix}
         A_{11} & A_{12} & A_{13} & A_{14} \\
         A_{21} & A_{22} & A_{23} & A_{24} \\
         A_{31} & A_{32} & A_{33} & A_{34} \\
         A_{41} & A_{42} & A_{43} & A_{44} \\
         \hline
         A_{51} & A_{52} & A_{53} & A_{54} \\
         \vdots & \vdots & \vdots & \vdots \\
         A_{(5+J-1)1} & A_{(5+J-1)2} & A_{(5+J-1)3} & A_{(5+J-1)4} \\
    \end{bmatrix} \in \mathbb{R}^{(1+k+k+1+J[k+2]) \times (J+k+1+k)}
\end{align*}
where
\begin{align*}
    A_{12} = \mathbf{1}^\top, 
    \: A_{22} = -I_{k\times k}, \: A_{24} = L, \: A_{32} = -I_{k\times k}, \: A_{43} = -1,
\end{align*}
and the part below the dashed line corresponds to the $J$ second-order cone constraints, where
\begin{align*}
    A_{(5+j-1)1}[:,j] &= (-1,\, 0, \, ...,\, 0,\, +1)  \in \mathbb{R}^{(2+k) \times 1} \quad \forall j\in \{1,...,J\} \\
    A_{(5+j-1)3} &= (-4,\, 0, \, ...,\, 0,\, -4)  \in \mathbb{R}^{(2+k) \times 1} \quad \forall j\in \{1,...,J\} \\
    A_{(5+j-1)4} &= \begin{bmatrix}
        0 & \cdots & 0 \\
        -2 & & 0 \\
         & \ddots & \\
         0 & & -2 \\
         0 & \cdots & 0
    \end{bmatrix} \in \mathbb{R}^{(2+k) \times k} \quad \forall j\in \{1,...,J\}
\end{align*}
and $X[:,j]$ denotes the $j$-th column of $X$. The rest of $A$ are zeros.

\subsubsection{Gaussian Reference Distribution}
\label{app:dro_cp_reformulation:portfolioopt:gaussianref}
Next, we turn attention to the setting where the nominal distribution is Gaussian rather than empirical, i.e., $\hat{\mathbb{P}} = \mathcal{N}(\hat{\mu}, \hat{\Sigma})$ for $\hat{\mu}, \hat{\Sigma}$ known. In this setting, by extending \cite[Theorem~5]{nguyen2023meancovariancerobustriskmeasurement} to the parametrized transportation cost $\kappa(\xi_1, \xi_2; \theta) = \|L^\top(\xi_1 - \xi_2)\|_2^2$ with $\theta = \{L \in \mathbb{L}_{++}^{d=k}\}$, we obtain
{
\begin{equation}\label{eq:DRO:mom}
    \begin{aligned}
    & \min_{w\in\mathcal{W}} \max_{\mathbb{Q} \in \mathcal{B}_\varepsilon (\hat{\mathbb{P}})} \mathcal{R}_{\xi\sim\mathbb{Q}}\left( -w^\top\xi \right)\\
     = & \min_{w\in\mathcal{W}} -\hat{\mu}^\top w + \alpha\sqrt{w^\top \hat{\Sigma} w} + \varepsilon \sqrt{1 + \alpha^2} \sqrt{w^\top (LL^\top)^{-1}w},
    \end{aligned}
\end{equation}}%
where $\alpha$ is a risk coefficient that only depends on $\gamma$, thus it is a constant. To bring \eqref{eq:DRO:mom} into the conic standard form \eqref{conic_primal_1}, we introduce auxiliary scalar variables $u$, $v$, and vector variables $z$ and $q$. We obtain:
\begin{equation*}
    \begin{aligned}
        \min_{w, u, v, q, z} \quad & -\hat{\mu}^\top w + \alpha u + \varepsilon \sqrt{1 + \alpha^2} v \\
        \textrm{subject to} \quad & \norm{z}_2 \leq u \\
        & \norm{q}_2 \leq v\\
        & w = Lq \\
        & \sqrt{\hat{\Sigma}} w = z \\
        & \mathbf{1}^\top w = 1\\
        & w \succeq 0,
    \end{aligned}
\end{equation*}
which can now be readily cast into the standard form \eqref{standard:conic} using the conic slack vector 
\begin{align*}
    s=\left(s_1, \bar{w}, \left(\bar{u}, z\right), \left(\bar{v}, \bar{q}\right) \right) \in \mathcal{K} := \{0\}^{k+1} \times \mathbb{R}^k_+ \times \mathrm{SOC}(k+1) \times\mathrm{SOC}(k+1)
\end{align*}
and the primal variable vector $x=(w, u, v, q) \in \mathbb{R}^{k + 1 + 1 + k}$. The corresponding problem data matrix and vectors are given by
\begin{align}
    {c} &= \left(-\hat{\mu}, \, \alpha, \, \varepsilon \sqrt{1+\alpha^2}, \, 0\, ...\, 0 \right) \\
    {b} &= ( 1, \, 0\, ...\, 0  ) \\
    {A} &= \begin{bmatrix}
        \mathbf{1}^\top & 0 & 0 & 0 \\
        -I_k & 0 & 0 & L \\
        -I_k & 0 & 0 & 0 \\
        0 & -1 & 0 & 0 \\
        -\sqrt{\hat{\Sigma}} & 0 & 0 & 0 \\
        0 & 0 & -1 & 0 \\
        0 & 0 & 0 & -I_k
    \end{bmatrix}. \label{eq:gaussian_primal_dro_reform_A_matrix}
\end{align}

\paragraph{Worst-case moments.}
\label{app:dro_cp_reformulation:portfolioopt:wcmoments}
Analogously to \cite[Theorem 5]{nguyen2023meancovariancerobustriskmeasurement}, we can then compute the worst-case moments as follows:
\begin{align*}
\mu^\star_L &= \hat{\mu} - \frac{\rho}{\sqrt{1+\alpha^2} \norm{w}_{(LL^\top)^{-1}}} (L L^\top)^{-1} w,
\end{align*}
\begin{align*}
\Sigma^\star_L &= \left(I + \frac{\lambda^\star (L L^\top)^{-1} w w^\top }{\gamma^\star - \lambda^\star w^\top (L L^\top)^{-1} w}\right) \hat{\Sigma} \left(I + \frac{\lambda^\star w w^\top (L L^\top)^{-1}}{\gamma^\star - \lambda^\star w^\top (L L^\top)^{-1} w} \right),
\end{align*}
where
\begin{equation*}
\gamma^\star = \frac{1}{2\rho} \sqrt{1+\alpha^2} \norm{w}_{(LL^\top)^{-1}}
\end{equation*}
and
\begin{equation*}
\lambda^\star = \left( \frac{w^\top (L L^\top)^{-1} w}{\gamma^\star} + \frac{2}{\alpha} \sqrt{w^\top \hat{\Sigma} w}  \right)^{-1}.
\end{equation*}

\subsection{Reformulation of the Linear Regression problem}
\label{app:dro_cp_reformulation:linreg}
In this section we provide convex reformulations for the distributionally robust linear regression problem with parametrized transportation cost in different settings.

\subsubsection{type-1 Mahalanobis distance and Absolute Error}
\label{app:dro_cp_reformulation:linreg:1type}
We provide here a reformulation of the distributionally robust linear regression problem under an $\ell_1$-loss function (see definition in Section \ref{numerics:linear:regression}) and show that it can be casted as a conic program. \rz{This loss is generally more forgiving (hence, robust) to large residuals than the squared loss, and therefore makes an interesting case study}. We further consider a parametrized transportation cost of the form $\kappa(\xi_1, \xi_2; \theta) = \|L^\top(\xi_1 - \xi_2)\|_2$ with $\theta = L \in \mathbb{L}_{++}^d$. By following a similar reasoning as in \cite[Theorem~4.2.1 (p.73)]{chen2020distributionally}, we obtain the following equivalence:
\begin{equation}\label{eq:regression:absolute}
    \min_{w\in\mathbb{R}^k} \max_{\mathbb{Q} \in \mathcal{B}_\varepsilon (\hat{\mathbb{P}})} \mathbb{E}_{\xi\sim\mathbb{Q}}\left( \ell_1(w,\xi) \right) = \min_{w\in\mathbb{R}^k} \frac{1}{J} \sum_{j=1}^J \abs{ (-w,1)^\top \hat{\xi}_j} + \varepsilon \norm{(-w,1)}_{\left(LL^\top\right)^{-1}},
\end{equation}%
with $w\in\mathbb{R}^k$ and $\hat{\xi}_j\in\mathbb{R}^{d=k+1}$. We can write \eqref{eq:regression:absolute} equivalently as:
\begin{equation*}
    \begin{aligned}
        \min \quad & \frac{1}{J} \mathbf{1}^\top v + \varepsilon u \\
        \textrm{subject to} \quad & a_j = v_j - y_j + w^\top x_j \\
        & b_j = v_j + y_j - w^\top x_j\\
        & \norm{z}_2 \leq u\\
        & a \geq 0, b \geq 0 \\
        & Lz + (w,0) = (0,1).
    \end{aligned}
\end{equation*}
We can now readily cast the above formulation in the standard form \eqref{standard:conic} using the conic slack vector
\begin{align*}
s=(0, a, b, (\Bar{u}, \Bar{z})) \in \mathcal{K} := \{0\}^{d} \times \mathbb{R}^{J}_+ \times \mathbb{R}^{J}_+ \times \mathrm{SOC}(d+1)
\end{align*}
and the primal variable vector $x=(v, u, z, w) \in \mathbb{R}^{J+1+d+k}$. The corresponding problem data matrix and vectors are
\begin{align*}
c &= \left(1/J, \, ...,\, 1/J,\, \varepsilon,\,0\, ...\, 0 \right) \\
b &= \left(0,\,...,\, 0,\,1, \,-y,\,y,\, 0\, ...\, 0 \right) \\
A &= \begin{bmatrix}
0 & 0 & L & [I_k,\,0] \\
-I_J & 0 & 0 & -X \\
-I_J & 0 & 0 & X \\
0 & -e_1 & [0,-I_d]^\top & 0 \\
\end{bmatrix}
\end{align*}
where $X = [x_1, \ldots, x_J]^\top \in \mathbb{R}^{J\times d}$ is the data matrix and $e_1$ is the first standard basis vector in $\mathbb{R}^d$.

Observe that \eqref{eq:regression:absolute} can be interpreted as a regularized linear regression, where the regression coefficients are penalized via $\varepsilon$ and $L$. Specifically, the regularizer could be seen as a control over the amount of ambiguity in the data and provides a rigorous theoretical foundation on why the $\ell_2$-regularizer prevents overfitting of the training data. The connection between robustness and regularization has been established in several works \cite{chen2020distributionally}.

\subsubsection{type-2 Mahalanobis distance and Squared Error}
\label{app:dro_cp_reformulation:linreg:1type:squared}
In this section, we provide the reader with the reformulation of the distributionally robust linear regression when using the squared error $\ell_2(w,\xi) = \left((-w,1)^\top \xi\right)^2$ and the parametrized transportation cost $\kappa(\xi_1, \xi_2; \theta) = \|L^\top(\xi_1 - \xi_2)\|_2^2$ with $\theta = L \in \mathbb{L}_{++}^d$. By invoking \cite[Proposition~2]{blanchet_robust_2019} , we have the following equivalence: 
\begin{equation*}
    \min_{w\in\mathbb{R}^k} \max_{\mathbb{Q} \in \mathcal{B}_\varepsilon (\hat{\mathbb{P}})} \mathbb{E}_{\xi\sim\mathbb{Q}}\left( \ell_2(w,\xi) \right) = \min_{w\in\mathbb{R}^k} \left( \sqrt{\frac{1}{J} \sum_{j=1}^J \left[ (-w,1)^\top \hat{\xi}_j \right]^2 } + \varepsilon \norm{(-w,1)}_{\left(LL^\top\right)^{-1}}  \right)^2
\end{equation*}%
The final reformulation, after recognizing nonnegativity of the terms in the square and thus optimizing its square root, follows as
\begin{equation*}
    \begin{aligned}
    \min \quad & \lambda/\sqrt{J} + \varepsilon z \\
    \textrm{subject to} \quad & \norm{q}_2 \leq z, \norm{a}_2 \leq \lambda\\
    & Lq = v \\
    & [X,y]\; v = a \\
    & v_d = 1,
    \end{aligned}
\end{equation*}
where $v_d$ is the $d$-th entry of the vector $v$. This can now be readily cast into the standard form \eqref{standard:conic} using the conic slack vector 
\begin{align*}
s=(0, (\bar{\lambda}, \bar{a}), (\bar{z}, \bar{q})) \in \mathcal{K} := \{0\}^{d+1+J} \times \mathrm{SOC}(J+1) \times \mathrm{SOC}(d+1)
\end{align*}
and the primal variable vector $x=(v, \lambda, a, z, q) \in \mathbb{R}^{d+1+J+1+d}$. The corresponding problem data matrix and vectors are
\begin{align*}
    {c} &= \left(0, 1/\sqrt{J}, \, 0\, ...\, 0,\,\varepsilon , \, 0\, ...\, 0 \right) \\
    {b} &= e_{d+1} \\
    {A} &= \begin{bmatrix}
        -I_d & 0 & 0 & 0 & L \\
        [0 ... 0\, 1] & 0 & 0 & 0 & 0 \\
        [X,y] & 0 & -I_J & 0 & 0\\
        0 & -1 & 0 & 0 & 0 \\
        0 & 0 & -I_J & 0 & 0 \\
        0 & 0 & 0 & -1 & 0 \\
        0 & 0 & 0 & 0 & -I_d
    \end{bmatrix},
\end{align*}
where $e_{d+1}$ is the unit vector with $1$ in entry $d+1$. The resulting objective value must then be squared to recover the worst-case expectation.

%% file: Contents/Appendix_A2_path_diff.tex
We provide here a brief definition of definable functions and sets, and refer the reader to the monograph \cite{coste1999introduction} for additional information.
\begin{definition}[{\cite[Definition 1.4-1.5]{coste1999introduction}}]
An o-minimal structure expanding the real closed field $\R$ is a collection $\mathcal{S}=(\mathcal{S}^n)_{n\in\N}$, where each $\mathcal{S}^n \subset \R^n$ satisfies the following
\begin{enumerate}
    \item all algebraic subsets of $\R^n$ are contained in $\mathcal{S}^n$.
    \item $\mathcal{S}^n$ is a Boolean subalgebra of $\R^n$.
    \item If $A\in \mathcal{S}^m$ and $B\in \mathcal{S}^m$, then $A \times B \in \mathcal{S}^{n+m}$.
    \item The projection onto the first $n$ coordinates of any $A\in \mathcal{S}^{n+1}$ belongs to $\mathcal{S}^n$.
    \item The elements of $\mathcal{S}^1$ are precisely the finite unions of points and intervals.
\end{enumerate}
The elements of $\mathcal{S}^n$ are called definable subsets of $\R^n$. A function is called definable if its graph is a definable set. Definable functions comprise the vast majority of functions commonly found in the field of optimization.
\end{definition}
The class of definable sets is large. In particular, all cones that are generally considered in the context of distributionally robust optimization (i.e., the exponential cone, the second-order cone, and the positive semidefinite cone) are all definable. Similarly, definable functions include almost all functions that are commonly found in the field of optimization \cite{bolte2021conservative}.

%% file: Contents/Appendix_B_differentiating_conic_program.tex
Conic programming generally refers to problems where the feasible set is the intersection of an affine subspace and a nonempty closed convex cone \cite{agrawal2020differentiatingconeprogram,odonoghue2016conicoptimizationoperatorsplitting,busseti2018solutionrefinementregularpoints}. 
In primal form, a conic program can be written as
\begin{equation}
    \begin{aligned}
        \min_{x,s} \quad & c^\top x \\
        \textrm{subject to} \quad & Ax + s = b \\
        & s \in \mathcal{K}
    \end{aligned}
    \label{conic_primal_1}
\end{equation}
where $x \in \mathbb{R}^n$, $s \in \mathbb{R}^m$ is a primal slack variable, and $\mathcal{K} \subseteq \mathbb{R}^m$ is a nonempty, closed, convex cone with dual cone $\mathcal{K}^* = \left\{ y \in \mathbb{R}^m \; \vert \; \inf_{z\in\mathcal{K}}y^\top z \geq 0 \right\} \subseteq \mathbb{R}^m$.

The dual of \eqref{conic_primal_1} is a conic program of the form
\begin{equation}
    \begin{aligned}
        \min_{y} \quad & b^\top y \\
        \textrm{subject to} \quad & A^\top y + c = 0 \\
        & y \in \mathcal{K}^*
    \end{aligned}
    \label{conic_dual_1}
\end{equation}
where $y \in \mathbb{R}^m$ is the dual variable.

In this section, we define the solution map $\mathcal{S}$ mapping the problem definition $(A,b,c)$ to the primal-dual solution $(x,y,s)$. Moreover, following \cite{bolte2021nonsmooth}, we provide sufficient conditions under which $\mathcal{S}$ admits a conservative Jacobian and provide an expression for it.

\subsection{Necessary optimality conditions}\label{ch:kkt}
Any primal-dual optimizers $(x,s,y)$ must satisfy the KKT conditions
\begin{subequations}
\begin{align}
    0 &= A^\top y + c, \\
    s &= -A x + b, \\
    0 &= s^\top y. \label{eq:zero_duality_gap}
\end{align}\label{eq:KKT_conditions}
\end{subequations}
Following \cite{bolte2021nonsmooth}, we can equivalently express \cref{eq:KKT_conditions} via
\begin{subequations}
\label{eq:kkt_equivalent}
\begin{align}
A^\top \Pi_{\mathcal{K}^*}(v)+c&=0,\\
-Au+v-\Pi_{\mathcal{K}^*}(v)+b&=0,
\end{align}
\end{subequations}
where $v=y-s$, $u=x$, and $\Pi_{\mathcal{K}^*}$ denotes the projector to the closed convex cone $\mathcal{K}^*$. Letting $z=(u,v)$, condition \cref{eq:kkt_equivalent} is equivalent to
\begin{align*}
\mathcal{N}(z,A,b,c):=(Q(A,b,c)-I)\Pi(z)+V(b,c)+z=0,
\end{align*}
where $\Pi$ denotes the projection onto $\R^n \times \mathcal{K}^*$, and
\begin{align*}
Q(A,b,c)=\begin{bmatrix}
0&A^\top\\-A&0
\end{bmatrix}\in\R^{N\times N},~~
V(b,c)=\begin{bmatrix}c\\b\end{bmatrix}\in\R^N,
\end{align*}
with $N=m+n$.

\subsection{The solution map}
The solution map $\mathcal{S}: (A,b,c) \mapsto (x,y,s)$ is defined as the mapping from the optimization problem data $(A, b, c)$ to the vectors $(x,y,s)$ that satisfy the KKT conditions of the conic optimization problem, assuming such vectors are unique.

One way to define $\mathcal{S}$ is through the following composition of functions \cite{bolte2021nonsmooth}:
\begin{align}
\mathcal{S}(A,b,c)=[\phi\circ \nu](A,b,c),  \label{eq:sol_map_composition}  
\end{align}
where
\begin{itemize}
    \item $\nu:\R^{m \times n} \times \R^m \times \R^n\to \R^N$ is implicitly defined as $\mathcal{N}(\nu(A,b,c),A,b,c)=0$;
    \item $\phi:\R^N\to\R^n\times\R^m\times\R^m$ is defined as $\phi(u,v)=(u,\Pi_{\mathcal{K}^*}(v),\Pi_{\mathcal{K}^*}(v)-v)$.
\end{itemize}

\subsection{Derivative of the Solution Map}
\label{ch:derivativesolutionmap}

The Jacobian of $\mathcal{S}$ can be obtained by computing the Jacobians of $\nu$ and $\phi$, and applying the chain rule of differentiation to \cref{eq:sol_map_composition}.

\begin{proposition}[{\cite[Proposition 4]{bolte2021nonsmooth}}]\label{proposition:diff}
Assume $\Pi_{\mathcal{K}^*}$ is locally Lipschitz and definable with convex conservative Jacobian $\J_{\Pi_{\mathcal{K}^*}}$. Let $\J_\mathcal{N}$ be the convex conservative Jacobian of the residual $\mathcal{N}$. Assume that given any $z=\nu(A,b,c)$, and $[J_z,J_A,J_b,J_c]\in\J_{\mathcal{N}}(z,A,b,c)$, all the matrices $J_z$ are invertible. Then $\mathcal{S}$ is locally Lipschitz and definable with conservative Jacobian $\J_{\mathcal{S}}(A,b,c)=\J_{\phi}(\nu(A,b,c))\J_\nu(A,b,c)$, where
\begin{align*}
\J_\nu(A,b,c)=\{ -U^{-1}V:[U~V]\in\J_\mathcal{N}(\nu(A,b,c),A,b,c)\},~~
\J_{\phi}(z)=\begin{bmatrix}
I&0\\0&\J_{\Pi_{\mathcal{K}^*}}(v)\\0& \J_{\Pi_{\mathcal{K}^*}}(v)-I
\end{bmatrix}.
\end{align*}
\end{proposition}
The residual map $\mathcal{N}$ can be obtained by simple addition and product operations starting from $(A,b,c)$ and $\Pi_{\mathcal{K}^*}$; therefore, it is always locally Lipschitz and definable in $(A,b,c)$ if the same holds for the projector $\Pi_{\mathcal{K}^*}$. We now prove that $\Pi$ is also locally Lipschitz and definable under the following, mild conditions.
\begin{assumption}\label{ass:K_definable}
The set $\mathcal{K}$ is definable in an o-minimal structure.
\end{assumption}
\begin{assumption}\label{ass:metric_regularity}
The primal \cref{conic_primal_1} and the dual \cref{conic_dual_1} satisfy the metric regularity condition
\begin{align}
b\in \operatorname*{int}\{ \operatorname*{span}A + \mathcal{K} \} ,~~~ c \in \operatorname*{int}\{ A^{\top} \mathcal{K}^* \}. \label{eq:metric_regularity_primal_dual}
\end{align}
\end{assumption}
Metric regularity ensures zero duality gap \cite[Theorem 3.6]{bonnans2013perturbation}, meaning that the KKT conditions are necessary and sufficient conditions of optimality. If $\operatorname*{int}\mathcal{K}$ is nonempty, then metric regularity of the primal problem is equivalent to \emph{Slater's constraint qualification}, i.e., to the existence of some $\bar{x}\in \mathcal{Q}$ such that $G \bar{ x} \in \operatorname*{int} \mathcal{K}$ \cite[Proposition 2.106]{bonnans2013perturbation}. Slater's constraint qualification is a mild assumption which is \emph{generic} in linear conic programs like \cref{conic_primal_1}, that is, it holds for almost all combinations of problem parameters \cite{dur2012slater}.
\begin{lemma}\label{lemma:proj_is_path_diff}
Under \cref{ass:K_definable,ass:metric_regularity}, let $(x,s,y)$ be a primal-dual optimizer of \cref{conic_primal_1}, and let $z=(x,y-s)$. Then $\Pi$ is locally Lipschitz and definable at $z$ and $\Pi(z)=(x,y)$.
\end{lemma}
\begin{proof}
Under metric regularity, $z$ solves the primal-dual embedding, meaning that $y\in \mathcal{K}^*$. Since $\Pi(z)=\Pi(x,y-s)=(x,\Pi_{\mathcal{K}^*}(y-s))$, and the map $z\mapsto x$ is linear, we only have to prove that the projector $\Pi_{\mathcal{K}^*}$ is path-differentiable at $y-s$. The point $v=\Pi_{\mathcal{K}^*}(y-s)$ uniquely satisfies $y-s-v\in N_{\mathcal{K}^*}(v)$, $v\in \mathcal{K}^*$. The set $N_{\mathcal{K}^*}(v)=\{ w: \langle w,z-v \rangle \leq 0,~ \forall z\in \mathcal{K}^* \}$ is definable, since it is defined by a first order condition \cite[Condition 4, page 12]{coste1999introduction}, and $\mathcal{K}^*$ is definable as an immediate consequence of \cref{ass:K_definable}. Therefore, $v$ is the unique element of a definable set, and the set $\{ (y,\Pi_{\mathcal{K}^*}(y)) : y\in \mathcal{K}^*\}$ is definable. This proves that $\Pi_{\mathcal{K}^*}$ is definable. Since projections to convex sets are Lipschitz, this concludes the proof.
\end{proof}

\begin{remark}
\cref{proposition:diff} provides sufficient conditions under which $\mathcal{S}$ is locally Lipschitz and definable in $(A,b,c)$. Since $(A,b,c)$ are parameterized in $\theta$, to fulfill the assumption on $x^\star(\theta)$ and $c$ in \cref{ass:path_diff_bilevel} we simply require that $(A,b,c)$ are Lipschitz continuous and definable functions of $\theta$, which is not a restrictive assumption.
\end{remark}
\subsubsection{Computing the Forward Derivative}
Computing the forward derivative of the solution map amounts to evaluating the product $J_\mathcal{S}(A,b,c)(\mathrm{d}A,\mathrm{d}b,\mathrm{d}c)$ between an element $J_\mathcal{S}(A,b,c)$ of the conservative Jacobian $\J_\mathcal{S}(A,b,c)$ of $\mathcal{S}$ and a direction $(\mathrm{d}A,\mathrm{d}b,\mathrm{d}c)$. Using \cref{proposition:diff} we have
\begin{equation*}
    (\mathrm{d}x, \mathrm{d}y, \mathrm{d}s) = J_\mathcal{S}(A,b,c)(\mathrm{d}A, \mathrm{d}b, \mathrm{d}c) = J_\phi(z)J_\nu(A,b,c)(\mathrm{d}A, \mathrm{d}b, \mathrm{d}c)
\end{equation*}
where $J_\phi(z)\in \J_{\phi}(z)$, and $J_\nu(A,b,c)\in\J_{\nu}(A,b,c)$ can be computed by solving
\begin{equation}
J_\nu(A,b,c) = \operatorname*{argmin}_{z} ~~ \| J_1x+J_2(\mathrm{d}A, \mathrm{d}b, \mathrm{d}c) \|^2, ~~ [J_1,J_2]\in\J_{\mathcal{N}}(z,A,b,c). \label{eq:app:differentiating:forward}
\end{equation}
Letting $J_\nu(A,b,c)=(\mathrm{d}u,\mathrm{d}v)$, the forward derivative $(\mathrm{d}x,\mathrm{d}y,\mathrm{d}s)$ of the primal-dual solution map can be obtained as follows
\begin{equation*}
    \begin{bmatrix}
        \mathrm{d}x \\
        \mathrm{d}y \\
        \mathrm{d}s
    \end{bmatrix} = \begin{bmatrix}
        \mathrm{d} u \\
        J_{\Pi_{\mathcal{K}^*}}(\beta) \mathrm{d}v \\
        J_{\Pi_{\mathcal{K}^*}}(\beta) \mathrm{d}v - \mathrm{d}v
    \end{bmatrix}.
\end{equation*}

\subsubsection{Computing the Adjoint Derivative}
For reverse auto-differentiation, we are interested in computing the product between $J_\mathcal{S}(A,b,c)^\top$, with $J_\mathcal{S}(A,b,c)\in\J_{\mathcal{S}}(A,b,c)$ and a direction $(\mathrm{d}x,\mathrm{d}y,\mathrm{d}s)$
\begin{equation}
    (\mathrm{d}A, \mathrm{d}b, \mathrm{d}c) = J_\mathcal{S}(A,b,c)^\top (\mathrm{d}x, \mathrm{d}y, \mathrm{d}s) = J_\nu^\top(A,b,c)^\top J_\phi(z)^\top (\mathrm{d}x, \mathrm{d}y, \mathrm{d}s). \label{eq:backward_diff}
\end{equation}
To evaluate \cref{eq:backward_diff}, we first compute
\begin{equation*}
    \mathrm{d}z = J_\phi(z)^\top (\mathrm{d}x, \mathrm{d}y, \mathrm{d}s) = \begin{bmatrix}
        \mathrm{d}x \\
        J_{\Pi_{\mathcal{K}^*}}^\top (v) (\mathrm{d}y + \mathrm{d}s) - \mathrm{d}s
    \end{bmatrix}.
\end{equation*}
Next, we can obtain $\mathrm{d}s=J_\mathcal{S}(A,b,c)^\top \mathrm{d}z$ by solving
\begin{equation}
\mathrm{d}s=\operatorname*{argmin}_{ds}~~\| J_1 ds + J_2 \mathrm{d}z \|^2,~~~ [J_1~J_2]\in\J_{\mathcal{N}}(z,A,b,c). \label{eq:app:differentiating:backward}
\end{equation}
%
%
%

%% file: Contents/Appendix_C_differentiate_Wassertein.tex
This appendix shows how a parametrized optimal transport problem can be computed and differentiated. To this end, we focus as an example on the type-p Mahalanobis distance with transportation cost $\kappa(\xi_1, \xi_2; \theta) = \|L^\top(\xi_1, \xi_2)\|_2^p$ with $\theta = L \in \mathbb{L}_{++}^d$. Other possible transportation cost parametrizations follow similarly. In particular, we will consider distance calculations between two discrete distributions (Appendix~\ref{ch:discrete_param_wassdist}) and between two Gaussian distributions (Appendix~\ref{ch:gaussian_param_wassdist}).

\subsection{Parametrized Mahalanobis distance between discrete distributions}
\label{ch:discrete_param_wassdist}
Given two discrete distributions, it is possible to calculate their parametrized Mahalanobis distance efficiently by solving a linear program.
\begin{fact}[Calculation of the discrete Mahalanobis distance]
    The parametrized type-$p$ Mahalanobis distance $d(\mathbb{P}, \mathbb{Q}; \theta)$ between two discrete distributions $\mathbb{P} \triangleq \sum_{i=1}^I p_i \delta_{x_i}$ and $\mathbb{Q} \triangleq  \sum_{j=1}^J q_j \delta_{y_j}$ can be calculated as the $p$-th root of the optimal value of the following linear program:
    \begin{equation}
        \begin{aligned}
            d(\mathbb{P}, \mathbb{Q}; \theta)^p = \min_{\pi_{ij}} \quad & \sum_{i=1}^{I} \sum_{j=1}^{J} \norm{L^\top (x_i - y_j)}_2^p \pi_{ij} =: \Tilde{c}(L)^\top \Tilde{\pi}^\star\\
            \emph{subject to} \quad & \sum_{i=1}^{I} \pi_{ij} = q_j \quad \forall j \in [J] \\
            & \sum_{j=1}^{J} \pi_{ij}= p_i \quad \forall i \in [I] \\
            & \pi_{ij} \geq 0 \quad \forall i \in [I],\; j \in [J].
        \end{aligned}
    \end{equation}
    \label{fac:calc_discrete_wassdist}
\end{fact}
The derivative of the parametrized Wasserstein distance with respect to the cost matrix $L$ between two discrete Distributions is also calculated efficiently as follows.
\begin{fact}[Differentiation of the discrete Mahalanobis distance]
    Suppose the solution map of \cref{fac:calc_discrete_wassdist} is differentiable at $(\mathbb{P},\mathbb{Q})$. By the envelope theorem \cite{milgrom2002envelope}, the gradient of the type-$p$ Mahalanobis distance between two discrete distributions $\mathbb{P} \triangleq  \sum_{i=1}^I p_i \delta_{x_i}$ and $\mathbb{Q} \triangleq  \sum_{j=1}^J q_j \delta_{y_j}$ with respect to the parameter matrix $L\in\mathbb{L}^d_{++}$ is given by
    \begin{equation}
        \pdv{d(\mathbb{P}, \mathbb{Q}; \theta)^p}{L_{k l}} = \sum_{r=1}^{I\cdot J} \pdv{\Tilde{c}_r(L)}{L_{k l}} \Tilde{\pi}_r^\star(L).
    \end{equation}
    By defining $\Delta_r = (x_i-y_j)$ for all combinations $r$ of $i\in I$ and $j\in J$ it is possible to write 
    \begin{equation}
        \pdv{c_r(L)}{L} = \frac{\Delta_r\Delta_r^\top L}{c_r(L)}.
    \end{equation}
    \label{fac:diff_discrete_wassdist}
\end{fact}
The path-differentiability assumption in \cref{ass:path_diff_bilevel} is fulfilled under the assumptions of \cref{proposition:diff}. Practically, $d$ is almost everywhere differentiable and its conservative Jacobian $\J_{d}$ of $d(\cdot, \cdot \; \theta)$ is almost everywhere equal to the gradient of $d$, which can be computed through Fact~\ref{fac:diff_discrete_wassdist}.

\subsubsection{Nonconvexity of the parametrized discrete Mahalanobis distance} \label{ch:convexity_param_discrete_wassdist}
We show that the parametrized discrete Mahalanobis distance is nonconvex with respect to the parameter matrix $L$ by providing counterexamples for the case $p=1$ and $p=2$. For simplicity, let $d_p(\mathbb{P},\mathbb{Q}; L)$ be the type-$p$ Wasserstein distance between the two discrete distributions $\mathbb{P}$ and $\mathbb{Q}$ with respect to the cost matrix $L$.

\paragraph{type-1 discrete Mahalanobis distance} Consider the following two discrete distributions:
\begin{align}
    &\mathbb{P} \triangleq  0.4 \delta_{x_1} + 0.6 \delta_{x_2} \text{ with } x_1 = (0.7,0.4) \text{ and } x_2=(1.7,1.0) \nonumber\\
    &\mathbb{Q} \triangleq  0.5 \delta_{y_1} + 0.5 \delta_{y_2} \text{ with } y_1 = (1.8,0.1) \text{ and } y_2=(0.5,1.4).\nonumber
\end{align}
Next, consider
\begin{align}
    L_1 = \begin{bmatrix}
        1 & 0\\
        0.5 & 0.5
    \end{bmatrix} \text{ and } L_2 = \begin{bmatrix}
        0.5 & 0\\
        1.0 & 1.0
    \end{bmatrix}\nonumber.
\end{align}
Solving the linear program in Fact~\ref{fac:calc_discrete_wassdist} yields
\begin{align}
    d_1(\mathbb{P},\mathbb{Q};L_1) &= 0.6203, \nonumber\\
    d_1(\mathbb{P},\mathbb{Q};L_2) &= 0.5036, \nonumber\\
    d_1(\mathbb{P},\mathbb{Q};0.5\cdot L_1 + 0.5 \cdot L_2) &= 0.6820, \nonumber
\end{align}
which indicates that
\begin{equation}
    d_1(\mathbb{P},\mathbb{Q};0.5\cdot L_2+0.5\cdot L_1) > 0.5\cdot d_1(\mathbb{P},\mathbb{Q};L_1) + 0.5\cdot d_1(\mathbb{P},\mathbb{Q};L_2), \nonumber
\end{equation}
thus proving the nonconvexity of the type-$1$ Mahalanobis distance.

\paragraph{type-2 discrete Mahalanobis distance} Consider the following two discrete distributions:
\begin{align}
    &\mathbb{P} \triangleq  0.6 \delta_{x_1} + 0.4 \delta_{x_2} \text{ with } x_1 = (1.2,1.9) \text{ and } x_2=(0.1,0.1) \nonumber \\
    &\mathbb{Q} \triangleq  0.6 \delta_{y_1} + 0.4 \delta_{y_2} \text{ with } y_1 = (0.2,1.4) \text{ and } y_2=(1.4,0.3). \nonumber
\end{align}
Next, consider
\begin{align}
    L_1 = \begin{bmatrix}
        1 & 0\\
        0.5 & 0.5
    \end{bmatrix} \text{ and } L_2 = \begin{bmatrix}
        0.5 & 0\\
        1.0 & 0.5
    \end{bmatrix}\nonumber.
\end{align}
Solving the linear program in Fact~\ref{fac:calc_discrete_wassdist} yields
\begin{align}
    d_2(\mathbb{P},\mathbb{Q};L_1) &= 1.0578, \nonumber\\
    d_2(\mathbb{P},\mathbb{Q};L_2) &= 0.9646, \nonumber\\
    d_2(\mathbb{P},\mathbb{Q};0.5\cdot L_1 + 0.5 \cdot L_2) &= 1.1433, \nonumber
\end{align}
which indicates that
\begin{equation}
    d_2(\mathbb{P},\mathbb{Q};0.5\cdot L_2+0.5\cdot L_1) > 0.5\cdot d_2(\mathbb{P},\mathbb{Q};L_1) + 0.5\cdot d_2(\mathbb{P},\mathbb{Q};L_2), \nonumber
\end{equation}
thus proving the nonconvexity of the type-$2$ Mahalanobis distance.

\subsection{Parametrized Gelbrich distance}
\label{ch:gaussian_param_wassdist}
We define the parametrized Gelbrich distance as follows.
\begin{definition}[Parametrized Gelbrich distance]
    The parameterized Gelbrich distance between two mean-covariance pairs $(\mu_\mathbb{Q},\;\Sigma_\mathbb{Q})$ and $(\mu_\mathbb{P},\;\Sigma_\mathbb{P})$ in $\mathbb{R}^d\times \mathbb{S}^d_+$, where $\mathbb{S}^d_+$ denotes the set of positive semidefinite symmetric matrices, is defined as
    \begin{align}
        g((\mu_\mathbb{Q},\;\Sigma_\mathbb{Q}), (\mu_\mathbb{P},\;\Sigma_\mathbb{P}); \, L) 
        &\coloneqq \Biggl( \norm{L^\top(\mu_\mathbb{Q} - \mu_\mathbb{P})}^2 + \Tr\left[  \left(\Sigma_\mathbb{Q} + \Sigma_\mathbb{P}\right) LL^\top \right]\nonumber \\
        &- 2 \Tr \left[\left( \Sigma_\mathbb{P}^{\frac{1}{2}} LL^\top \Sigma_\mathbb{Q} LL^\top \Sigma_\mathbb{P}^{\frac{1}{2}} \right)^{\frac{1}{2}} \right] \Biggr)^{\frac{1}{2}},
    \end{align}
    for the parameter matrix $L\in\mathbb{L}^d_{++}$.
    \label{def:calc_parametrized_gelbrich_distance}
\end{definition}
Note that for $\mathbb{P} \triangleq \mathcal{N}(\mu_\mathbb{P}, \Sigma_{\mathbb{P}})$ and $\mathbb{Q} \triangleq \mathcal{N}(\mu_\mathbb{Q}, \Sigma_{\mathbb{Q}})$, one has
    \begin{equation}
        g((\mu_\mathbb{Q},\;\Sigma_\mathbb{Q}), (\mu_\mathbb{P},\;\Sigma_\mathbb{P}); \, L) = d_2(\mathbb{Q},\mathbb{P}; L).
    \end{equation}
The local Lipschitz continuity and definability assumption in \cref{ass:path_diff_bilevel} is immediately verified by the Gelbrich distance since the definition in \cref{def:calc_parametrized_gelbrich_distance} involves compositions of definable functions (square root, trace, transposition, product). Moreover, since the Gelbrich distance is continuously differentiable everywhere, its conservative Jacobian coincides with its gradient \cite[Theorem 1]{bolte2021conservative}.

We can efficiently differentiate the parametrized Gelbrich distance with respect to the parameter matrix $L$ as follows. To shorten notation, we now write $C = (\mu_\mathbb{Q} - \mu_\mathbb{P})(\mu_\mathbb{Q} - \mu_\mathbb{P})^\top + \Sigma_\mathbb{Q} + \Sigma_\mathbb{P}$, $A = \sqrt{\Sigma_\mathbb{Q}}$, and $B = \sqrt{\Sigma_\mathbb{P}}$. As the derivative of the square root and the chain rule in the scalar case are well-known, we focus on the derivative of the following function in $L$
\begin{equation}
    h(L) = \Tr\left[ CLL^\top \right] - 2 \Tr \left[ \sqrt{ B LL^\top A A LL^\top B}\right]. \nonumber
\end{equation}
The derivative of the first term is $\pdv{}{L} \Tr\left[ CLL^\top \right] = CL + C^\top L$ \cite{Petersen2008}. To compute the derivative of the second term, let $f(L) = \Tr \left[ \sqrt{ B LL^\top A A LL^\top B}\right]$ and
{\footnotesize
\begin{align}
    &\mathcal{H}: \mathbb{R}^{d\times d} \rightarrow \mathbb{S}_{++}^d, \quad&& \mathcal{H}(X) = XX^\top, \quad&&& \partial H = \left( \partial X \right) X^\top + X \left( \partial X \right)^\top \nonumber \\
    &\mathcal{P}: \mathbb{S}_{++}^{d} \rightarrow \mathbb{R}^{d\times d}, \quad&& \mathcal{P}(X) = BXA, \quad&&& \partial P = B \left(\partial X\right) A \nonumber\\
    &\mathcal{S}: \mathbb{S}_{++}^{d} \rightarrow \mathbb{S}_{++}^{d}, \quad&& \mathcal{S}(X) = \sqrt{X}, \quad&&& \partial S \text{ satisfies: } \quad \left(\partial S\right) S + S \left(\partial S\right) = \partial X \nonumber\\
    &\mathcal{Y}: \mathbb{S}_{++}^{d} \rightarrow \mathbb{R}_+, \quad&& \mathcal{Y}(X) = \Tr\left[X\right], \quad&&& \partial y = \Tr\left[\partial X\right].\nonumber
\end{align}
}%
We then have
\begin{equation}
    f(L) = \mathcal{Y}\Bigl( \mathcal{S} \Bigl[ \mathcal{H} \bigl( \mathcal{P} \bigl[ \mathcal{H} \left(L\right) \bigr] \bigr) \Bigr] \Bigr) \quad \text{or} \quad f = \mathcal{Y}\; \circ \; \mathcal{S} \; \circ \; \mathcal{H} \; \circ \; \mathcal{P} \; \circ \; \mathcal{H}.\nonumber
\end{equation}
Working backwards, we define the following functions
\begin{align}
    g_1 &= \mathcal{Y}: \quad S \mapsto y \nonumber\\
    g_2 &= \mathcal{Y} \; \circ \; \mathcal{S}: \quad H_2 \mapsto y \nonumber\\
    g_3 &= \mathcal{Y} \; \circ \; \mathcal{S} \; \circ \; \mathcal{H}: \quad P \mapsto y\nonumber\\
    g_4 &= \mathcal{Y} \; \circ \; \mathcal{S} \; \circ \; \mathcal{H} \; \circ \; \mathcal{P}: \quad H_1 \mapsto y\nonumber\\
    g_5 &= \mathcal{Y} \; \circ \; \mathcal{S} \; \circ \; \mathcal{H} \; \circ \; \mathcal{P} \; \circ \; \mathcal{H} = f: \quad L \mapsto y\nonumber
\end{align}
Applying the chain rule yields
\begin{align}
    \textcolor{myorange}{\pdv{g_2(H_2)}{[H_2]_{ij}}} = \pdv{\mathcal{Y}(S)}{[H_2]_{ij}} = \pdv{\mathcal{Y}(\mathcal{S}(H_2))}{[H_2]_{ij}} &= \mathrm{Tr}\left[ (\pdv{\mathcal{Y}(S)}{S})^\top \pdv{\mathcal{S}(H_2)}{[H_2]_{ij}} \right]
    = \mathrm{Tr}\left[ \pdv{\mathcal{S}(H_2)}{[H_2]_{ij}} \right],\nonumber
\end{align}
where we used $\pdv{\mathcal{Y}(S)}{S} = I$ in the last equality. We then obtain
\begin{align}
    \textcolor{mygreen}{\pdv{g_3(P)}{[P]_{ij}}} = \pdv{g_2(H_2)}{[P]_{ij}} = \pdv{g_2(\mathcal{H}(P))}{[P]_{ij}} &= \mathrm{Tr}\left[ \textcolor{myorange}{\left(\pdv{g_2(H_2)}{H_2}\right)^\top} \pdv{\mathcal{H}(P)}{[P]_{ij}} \right]\nonumber
\end{align}
\begin{align}
    \textcolor{myviolet}{\pdv{g_4(H_1)}{[H_1]_{ij}}} = \pdv{g_3(P)}{[H_1]_{ij}} = \pdv{g_3(\mathcal{P}(H_1))}{[H_1]_{ij}} &= \mathrm{Tr}\left[ \textcolor{mygreen}{\left(\pdv{g_3(P)}{P}\right)^\top} \pdv{\mathcal{P}(H_1)}{[H_1]_{ij}} \right]\nonumber
\end{align}
\begin{align}
    \pdv{g_5(L)}{[L]_{ij}} = \pdv{g_4(H_1)}{[L]_{ij}} = \pdv{g_4(\mathcal{H}(L))}{[L]_{ij}} &= \mathrm{Tr}\left[ \textcolor{myviolet}{\left(\pdv{g_4(H_1)}{H_1}\right)^\top} \pdv{\mathcal{H}(L)}{[L]_{ij}} \right]\nonumber
\end{align}
At this point, we are left with the task of computing the elementary derivative matrices (in black). We can do so by recognizing the following identities, where $E^{ij}$ is a matrix such that $E^{ij}(i,j) = 1$ at index $(i,j)$ and zero everywhere else:
\begin{align}
    \pdv{\mathcal{H}(L)}{[L]_{ij}} &= E^{ij} L^\top + L \left(E^{ij}\right)^\top \nonumber \\
    \pdv{\mathcal{P}(H_1)}{[H_1]_{ij}} &= B E^{ij} A \nonumber\\
    \pdv{\mathcal{H}(P)}{[P]_{ij}} &= \left(E^{ij}\right)^\top P + P^\top E^{ij} \nonumber\\
    \left(\pdv{\mathcal{S}(H_2)}{[H_2]_{ij}}\right) S + S \left( \pdv{\mathcal{S}(H_2)}{[H_2]_{ij}} \right) &= E^{ij}.\nonumber
\end{align}
The matrices on the right-hand side above are available from the forward pass, and the partial derivative of $\mathcal{S}$ with respect to $\left[ H_2 \right]_{ij}$ is found by solving the corresponding Lyapunov equation in the last row. Finally, the gradient matrix is obtained by arranging the entries at their respective indices, as done in the backward step:
\begin{equation}
    \pdv{f(L)}{L} = \left[ \pdv{g_5(L)}{[L]_{ij}} \right].\nonumber
\end{equation}
The overall gradient of the Gelbrich distance results from elementary calculations. 

\subsubsection{Nonconvexity of the Parametrized Gelbrich Distance}
\label{ch:convexity_param_gelbrich_dist}
We provide examples showing the nonconvexity of the parametrized (squared) Gelbrich distance between two mean-covariance pairs, as defined in Definition \ref{ch:gaussian_param_wassdist}, with respect to the parameter $L$. Again, we proceed by offering a counterexample.

\paragraph{Gelbrich Distance.} Consider the two mean-covariance pairs
\begin{align}
    \mu_\mathbb{P} = (1.0, 0.6) &\text{ and } \Sigma_\mathbb{P} = \begin{bmatrix}
        0.1 & 0.0 \\
        0.0 & 1.0
    \end{bmatrix} \nonumber \\
    \mu_\mathbb{Q} = (0.8, 0.6) &\text{ and } \Sigma_\mathbb{Q} = \begin{bmatrix}
        10.0 & 0.0 \\
        0.0 & 1.0
    \end{bmatrix}, \nonumber
\end{align}
and the two $L$ matrices
\begin{align}
    L_1 = \begin{bmatrix}
        0.2 & 0\\
        0.2 & 1.9
    \end{bmatrix} \text{ and } L_2 = \begin{bmatrix}
        0.6 & 0\\
        0.8 & 0.5
    \end{bmatrix}.\nonumber
\end{align}
To shorten the notation, let $N_\mathbb{P}=(\mu_\mathbb{P},\Sigma_\mathbb{P})$ and $N_\mathbb{Q}=(\mu_\mathbb{Q},\Sigma_\mathbb{Q})$. Applying the formula in Definition~\ref{def:calc_parametrized_gelbrich_distance} yields
\begin{align*}
    g(N_\mathbb{P},N_\mathbb{Q};L_1) &= 0.5675, \\
    g(N_\mathbb{P},N_\mathbb{Q};L_2) &= 1.3142, \\
    g(N_\mathbb{P},N_\mathbb{Q};0.5\cdot L_1 + 0.5 \cdot L_2) &= 1.0636,
\end{align*}
indicating that
\begin{equation}
    g(N_\mathbb{P},N_\mathbb{Q};0.5\cdot L_2+0.5\cdot L_1) > 0.5\cdot g(N_\mathbb{P},N_\mathbb{Q};L_1) + 0.5\cdot g(N_\mathbb{P},N_\mathbb{Q};L_2), \nonumber
\end{equation}
and proving that the Gelbrich distance is not convex.

\paragraph{Squared Gelbrich Distance.} Similarly, consider the two mean-covariance pairs
\begin{align}
    \mu_\mathbb{P} = (0.4, 0.6) &\text{ and } \Sigma_\mathbb{P} = \begin{bmatrix}
        0.1 & 0.0 \\
        0.0 & 1.0
    \end{bmatrix} \nonumber \\
    \mu_\mathbb{Q} = (0.4, 0.4) &\text{ and } \Sigma_\mathbb{Q} = \begin{bmatrix}
        10.0 & 0.0 \\
        0.0 & 1.0
    \end{bmatrix}, \nonumber
\end{align}
and the two $L$ matrices
\begin{align}
    L_1 = \begin{bmatrix}
        0.7 & 0\\
        0.4 & 1.9
    \end{bmatrix} \text{ and } L_2 = \begin{bmatrix}
        0.9 & 0\\
        0.9 & 0.6
    \end{bmatrix}. \nonumber
\end{align}
Again, let $N_\mathbb{P}=(\mu_\mathbb{P},\Sigma_\mathbb{P})$ and $N_\mathbb{Q}=(\mu_\mathbb{Q},\Sigma_\mathbb{Q})$. Proceeding as before, we obtain
\begin{align}
    g(N_\mathbb{P},N_\mathbb{Q};L_1)^2 &= 3.9720, \nonumber\\
    g(N_\mathbb{P},N_\mathbb{Q};L_2)^2 &= 4.4900, \nonumber\\
    g(N_\mathbb{P},N_\mathbb{Q};0.5\cdot L_1 + 0.5 \cdot L_2)^2 &= 4.4865, \nonumber
\end{align}
indicating that
\begin{equation}
    g(N_\mathbb{P},N_\mathbb{Q};0.5\cdot L_2+0.5\cdot L_1)^2 > 0.5\cdot g(N_\mathbb{P},N_\mathbb{Q};L_1)^2 + 0.5\cdot g(N_\mathbb{P},N_\mathbb{Q};L_2)^2, \nonumber
\end{equation}
and proving that the squared Gelbrich distance is not convex.    

%% file: Contents/Appendix_D_numerical_details.tex
This appendix provides further details concerning the numerical experiments presented in Section~\ref{ch:numerics}. 

The complete pseudocode used in the numerical experiments is provided in Algorithm~\ref{alg:complete}. Compared to Algorithm~\ref{alg:learning_ambiguity} reported in Section~\ref{ch:algorithm}, Algorithm~\ref{alg:complete} enhances numerical stability by clipping the values of $J_\phi(\theta_i)$ between a lower bound $\underline{\nabla}_J$ and an upper bound $\overline{\nabla}_J$; similarly, it clips the eigenvalues of $\theta_{i+1}^\prime (\theta_{i+1}^\prime)^\top$ between $ \underline{\lambda}_M$ and $\overline{\lambda}_M$.

Table~\ref{tab:hyperparameters_general} reports the values of the hyperparameters used for all numerical procedures.


\begin{algorithm}[htbp]
    \renewcommand{\baselinestretch}{1.2}\selectfont
    \DontPrintSemicolon
    \SetAlgoLined
    \KwIn{initial guess $\theta_0 = I_d$, samples $\hat{\mathcal{D}}_J =\{\hat{\xi}_j\}_{j=1}^J$, step sizes $\{\alpha_i\}_{i\in\N}>0$, $\alpha_i>0$}
    \KwOut{$\theta^\star$, $\hat{w}_{\theta^\star}$}
    Bootstrap $n_b$ reference distributions $\{\hat{\mathbb{P}}_k\}_{k=1}^{n_b}$ by resampling $\hat{\mathcal{D}}_J$ with replacement\\
    Find $\varepsilon$ as the $1-\beta$ quantile of $\{d(\hat{\mathbb{P}}_k, \hat{\mathbb{P}}; \theta_0)\}_{k=1}^{n_b}$\\
    \For{$i \leftarrow 1$ \KwTo {\normalfont\texttt{maxiter}}}{
    Solve conic program with $A(\theta_i)$, $b(\theta_i)$, $c(\theta_i)$ for $z^\star(\theta_i)$ \\
    Differentiate through conic program for $J_\mathcal{S}(\theta_i)^\top c(\theta_i)$ \\
    Calculate $d(\hat{\mathbb{P}}_k, \hat{\mathbb{P}};\theta_i)$ for all $k\in \left[n_b\right]$ and obtain $J_e(\theta_i)\in\J_{e}(\theta_i)$\\
    Compute $J_\varphi(\theta_i) = J_\mathcal{S}(\theta_i)^\top c(\theta_i) + \mathcal{S}(\theta_i)^\top J_c(\theta_i) + 2\lambda_\mathrm{p} \max \{0,e(\theta_i)\} J_e (\theta_i)$\\
    Clip the values of $J_\varphi(\theta_i)$ between $\left[\underline{\nabla}_J, \overline{\nabla}_J\right]$\\
    Update $\theta^\prime_{i+1} = \theta_i - \bar{\alpha} \nabla J_\varphi(\theta_i)_\mathrm{clip}$\\
    Clip the eigenvalues of $M=\theta_{i+1}^\prime (\theta_{i+1}^\prime)^\top$ between $\left[ \underline{\lambda}_M, \, \overline{\lambda}_M \right]$\\
    Find $\theta_{i+1}$ as the lower triangular Cholesky factorization of $M$\\
    Update the conic program parameters $A(\theta_{i+1}), b(\theta_{i+1}), c(\theta_{i+1})$\\
       \If{$\left(\varphi(\theta_i) - \varphi(\theta_{i+1})\right) / \varphi(\theta_i) < {\normalfont\texttt{tol}}$}{
           \textbf{break}
       }
    }
    \caption{Loss-aware Distributionally Robust Optimization (complete)}
    \label{alg:complete}
\end{algorithm}

\begin{table}[htbp]
    \centering
    \begin{tabular}{l l}
    \hline
    Parameter & Value \\
    \hline
    $\alpha_i\equiv\bar{\alpha}$ & $1 \times 10^{-4}$ \\
    $\mathrm{tol}$ & $1 \times 10^{-6}$ \\
    $\lambda_\mathrm{p}$ & $10$ \\
    $\eta_\mathrm{p}$ & $100.0$ \\
    ${\normalfont\texttt{maxiter}}$ & $1 \times 10^6$ \\
    $\overline{\nabla}_J$ & $1000$ \\
    $\underline{\nabla}_J$ & $-1000$ \\
    $\overline{\lambda}_M$ & $1 \times 10^6$ \\
    $\underline{\lambda}_M$ & $1 \times 10^{-6}$ \\
    \hline
    \end{tabular}
    \vspace{0.2cm}
    \caption{List of Hyperparameters used for all numerical examples.}
    \label{tab:hyperparameters_general}
\end{table}

\subsection{Portfolio optimization example}
\label{app:numerics_details:portfolioopt}

We state the generation process of the Gaussian distributions used in Subsection~\ref{numerics:portfolioopt} and the means and covariances that resulted for the multi-experiment analysis. Note that each Gaussian distribution is resampled $10$ times to generate different datasets and thus strengthen the expressiveness of the results. For $k=3$, $\overline{\mu}=1.0$, $\underline{\mu}=-1.0$, $\overline{\sigma} = 0.1$, and $\underline{\sigma}=0.01$, the Gaussian distributions are generated as 
\begin{equation}
    \mu \sim \mathcal{U}(\underline{\mu},\overline{\mu})^k
\end{equation}
and
\begin{align}
    &\Tilde{\Sigma} \sim \mathcal{U}(\underline{\sigma},\overline{\sigma})^{k\times k}\\
    &\Sigma = \Tilde{\Sigma} \Tilde{\Sigma}^\top + 10^{-6} I_k. \label{eq:true_gaussian_cov_generation}
\end{align}

For \(i=1,\dots,50\), the \(i\)-th Gaussian distribution is defined as $\mathcal{N}(\mu_k, \Sigma_k)$, where \(\mu_i\) and  \(\Sigma_i\) are reported in Table~\ref{tab:experiments_dists_gaussian_cvar}.

\begin{table}
\centering
\caption{List of Gaussian distributions (means and covariance matrices) used in the experiments.}
\vspace{0.1cm}
{\scriptsize
\begin{minipage}{0.44\textwidth}
    \begin{tabular}{|c|c|c|}
    \hline
    \multirow{2}{*}{Dist.} & \multirow{2}{*}{Mean $\mu_i$} & \multirow{2}{*}{Covariance Matrix $\Sigma_i$} \\
                                   &                                  &                                    \\ \hline
    1 & $\begin{pmatrix}-0.99 \\ 0.74 \\ -0.51 \\ \end{pmatrix}$ & $\begin{pmatrix}0.010 & 0.014 & 0.010 \\ 0.014 & 0.020 & 0.015 \\ 0.010 & 0.015 & 0.012 \\ \end{pmatrix}$ \\ 
    2 & $\begin{pmatrix}-0.98 \\ 0.00 \\ -0.01 \\ \end{pmatrix}$ & $\begin{pmatrix}0.002 & 0.003 & 0.005 \\ 0.003 & 0.004 & 0.007 \\ 0.005 & 0.007 & 0.011 \\ \end{pmatrix}$ \\ 
    3 & $\begin{pmatrix}-0.90 \\ 0.01 \\ 0.04 \\ \end{pmatrix}$ & $\begin{pmatrix}0.006 & 0.006 & 0.007 \\ 0.006 & 0.006 & 0.007 \\ 0.007 & 0.007 & 0.009 \\ \end{pmatrix}$ \\ 
    4 & $\begin{pmatrix}-0.85 \\ 0.56 \\ -0.12 \\ \end{pmatrix}$ & $\begin{pmatrix}0.020 & 0.005 & 0.009 \\ 0.005 & 0.001 & 0.002 \\ 0.009 & 0.002 & 0.005 \\ \end{pmatrix}$ \\ 
    5 & $\begin{pmatrix}-0.75 \\ -0.48 \\ -0.19 \\ \end{pmatrix}$ & $\begin{pmatrix}0.012 & 0.009 & 0.010 \\ 0.009 & 0.011 & 0.007 \\ 0.010 & 0.007 & 0.010 \\ \end{pmatrix}$ \\ 
    6 & $\begin{pmatrix}-0.74 \\ -0.00 \\ 0.20 \\ \end{pmatrix}$ & $\begin{pmatrix}0.016 & 0.013 & 0.008 \\ 0.013 & 0.011 & 0.006 \\ 0.008 & 0.006 & 0.007 \\ \end{pmatrix}$ \\ 
    7 & $\begin{pmatrix}-0.68 \\ 0.14 \\ -0.25 \\ \end{pmatrix}$ & $\begin{pmatrix}0.005 & 0.004 & 0.007 \\ 0.004 & 0.007 & 0.009 \\ 0.007 & 0.009 & 0.013 \\ \end{pmatrix}$ \\ 
    8 & $\begin{pmatrix}-0.68 \\ -1.00 \\ -0.57 \\ \end{pmatrix}$ & $\begin{pmatrix}0.019 & 0.008 & 0.016 \\ 0.008 & 0.004 & 0.007 \\ 0.016 & 0.007 & 0.015 \\ \end{pmatrix}$ \\ 
    9 & $\begin{pmatrix}-0.64 \\ -0.20 \\ 0.79 \\ \end{pmatrix}$ & $\begin{pmatrix}0.004 & 0.005 & 0.006 \\ 0.005 & 0.008 & 0.009 \\ 0.006 & 0.009 & 0.010 \\ \end{pmatrix}$ \\ 
    10 & $\begin{pmatrix}-0.56 \\ 0.74 \\ -0.59 \\ \end{pmatrix}$ & $\begin{pmatrix}0.002 & 0.002 & 0.004 \\ 0.002 & 0.004 & 0.004 \\ 0.004 & 0.004 & 0.010 \\ \end{pmatrix}$ \\ 
    11 & $\begin{pmatrix}-0.53 \\ -0.14 \\ -0.82 \\ \end{pmatrix}$ & $\begin{pmatrix}0.004 & 0.008 & 0.008 \\ 0.008 & 0.015 & 0.016 \\ 0.008 & 0.016 & 0.018 \\ \end{pmatrix}$ \\ 
    12 & $\begin{pmatrix}-0.50 \\ 0.89 \\ -0.62 \\ \end{pmatrix}$ & $\begin{pmatrix}0.009 & 0.010 & 0.007 \\ 0.010 & 0.010 & 0.007 \\ 0.007 & 0.007 & 0.009 \\ \end{pmatrix}$ \\ 
    13 & $\begin{pmatrix}-0.44 \\ -0.08 \\ -0.76 \\ \end{pmatrix}$ & $\begin{pmatrix}0.003 & 0.005 & 0.005 \\ 0.005 & 0.014 & 0.012 \\ 0.005 & 0.012 & 0.011 \\ \end{pmatrix}$ \\ 
    14 & $\begin{pmatrix}-0.34 \\ -0.19 \\ 0.15 \\ \end{pmatrix}$ & $\begin{pmatrix}0.010 & 0.011 & 0.012 \\ 0.011 & 0.013 & 0.014 \\ 0.012 & 0.014 & 0.018 \\ \end{pmatrix}$ \\ 
    15 & $\begin{pmatrix}-0.32 \\ -0.08 \\ 0.87 \\ \end{pmatrix}$ & $\begin{pmatrix}0.016 & 0.015 & 0.012 \\ 0.015 & 0.014 & 0.011 \\ 0.012 & 0.011 & 0.010 \\ \end{pmatrix}$ \\ 
    16 & $\begin{pmatrix}-0.27 \\ 0.19 \\ -0.22 \\ \end{pmatrix}$ & $\begin{pmatrix}0.005 & 0.005 & 0.008 \\ 0.005 & 0.010 & 0.015 \\ 0.008 & 0.015 & 0.023 \\ \end{pmatrix}$ \\ 
    17 & $\begin{pmatrix}-0.27 \\ -0.60 \\ -0.82 \\ \end{pmatrix}$ & $\begin{pmatrix}0.006 & 0.010 & 0.006 \\ 0.010 & 0.016 & 0.011 \\ 0.006 & 0.011 & 0.013 \\ \end{pmatrix}$ \\ 
    18 & $\begin{pmatrix}-0.22 \\ 0.19 \\ 0.03 \\ \end{pmatrix}$ & $\begin{pmatrix}0.005 & 0.007 & 0.004 \\ 0.007 & 0.012 & 0.008 \\ 0.004 & 0.008 & 0.005 \\ \end{pmatrix}$ \\ 
    19 & $\begin{pmatrix}-0.20 \\ 0.43 \\ -0.44 \\ \end{pmatrix}$ & $\begin{pmatrix}0.013 & 0.010 & 0.009 \\ 0.010 & 0.011 & 0.008 \\ 0.009 & 0.008 & 0.008 \\ \end{pmatrix}$ \\ 
    20 & $\begin{pmatrix}-0.17 \\ 0.44 \\ -1.00 \\ \end{pmatrix}$ & $\begin{pmatrix}0.007 & 0.006 & 0.009 \\ 0.006 & 0.011 & 0.015 \\ 0.009 & 0.015 & 0.020 \\ \end{pmatrix}$ \\ 
    21 & $\begin{pmatrix}-0.16 \\ 0.85 \\ -0.45 \\ \end{pmatrix}$ & $\begin{pmatrix}0.006 & 0.007 & 0.007 \\ 0.007 & 0.015 & 0.010 \\ 0.007 & 0.010 & 0.008 \\ \end{pmatrix}$ \\ 
    22 & $\begin{pmatrix}-0.13 \\ -0.95 \\ 0.10 \\ \end{pmatrix}$ & $\begin{pmatrix}0.014 & 0.012 & 0.010 \\ 0.012 & 0.012 & 0.011 \\ 0.010 & 0.011 & 0.011 \\ \end{pmatrix}$ \\ 
    23 & $\begin{pmatrix}-0.11 \\ 0.14 \\ 0.82 \\ \end{pmatrix}$ & $\begin{pmatrix}0.017 & 0.012 & 0.009 \\ 0.012 & 0.015 & 0.011 \\ 0.009 & 0.011 & 0.008 \\ \end{pmatrix}$ \\ 
    24 & $\begin{pmatrix}-0.03 \\ -0.50 \\ 0.44 \\ \end{pmatrix}$ & $\begin{pmatrix}0.013 & 0.009 & 0.008 \\ 0.009 & 0.007 & 0.006 \\ 0.008 & 0.006 & 0.007 \\ \end{pmatrix}$ \\ 
    25 & $\begin{pmatrix}-0.02 \\ -0.53 \\ -0.86 \\ \end{pmatrix}$ & $\begin{pmatrix}0.011 & 0.011 & 0.012 \\ 0.011 & 0.014 & 0.012 \\ 0.012 & 0.012 & 0.015 \\ \end{pmatrix}$ \\ 
    \hline
    \end{tabular}
\end{minipage}
}%
\hfill
{\scriptsize
\begin{minipage}{0.43\textwidth}
    \begin{tabular}{|c|c|c|}
    \hline
    \multirow{2}{*}{Dist.} & \multirow{2}{*}{Mean $\mu_i$} & \multirow{2}{*}{Covariance Matrix $\Sigma_i$} \\
                                   &                                  &                                    \\ \hline
    26 & $\begin{pmatrix}0.00 \\ 0.08 \\ -0.73 \\ \end{pmatrix}$ & $\begin{pmatrix}0.003 & 0.006 & 0.005 \\ 0.006 & 0.017 & 0.011 \\ 0.005 & 0.011 & 0.016 \\ \end{pmatrix}$ \\ 
    27 & $\begin{pmatrix}0.10 \\ 0.43 \\ 0.21 \\ \end{pmatrix}$ & $\begin{pmatrix}0.016 & 0.011 & 0.009 \\ 0.011 & 0.009 & 0.006 \\ 0.009 & 0.006 & 0.005 \\ \end{pmatrix}$ \\ 
    28 & $\begin{pmatrix}0.10 \\ 0.42 \\ -0.42 \\ \end{pmatrix}$ & $\begin{pmatrix}0.017 & 0.006 & 0.010 \\ 0.006 & 0.002 & 0.003 \\ 0.010 & 0.003 & 0.012 \\ \end{pmatrix}$ \\ 
    29 & $\begin{pmatrix}0.13 \\ -0.14 \\ -0.81 \\ \end{pmatrix}$ & $\begin{pmatrix}0.002 & 0.005 & 0.005 \\ 0.005 & 0.020 & 0.016 \\ 0.005 & 0.016 & 0.015 \\ \end{pmatrix}$ \\ 
    30 & $\begin{pmatrix}0.15 \\ 0.06 \\ 0.53 \\ \end{pmatrix}$ & $\begin{pmatrix}0.002 & 0.003 & 0.004 \\ 0.003 & 0.004 & 0.006 \\ 0.004 & 0.006 & 0.012 \\ \end{pmatrix}$ \\ 
    31 & $\begin{pmatrix}0.30 \\ -0.91 \\ -0.96 \\ \end{pmatrix}$ & $\begin{pmatrix}0.018 & 0.008 & 0.013 \\ 0.008 & 0.004 & 0.006 \\ 0.013 & 0.006 & 0.010 \\ \end{pmatrix}$ \\ 
    32 & $\begin{pmatrix}0.39 \\ 0.63 \\ -0.31 \\ \end{pmatrix}$ & $\begin{pmatrix}0.008 & 0.004 & 0.008 \\ 0.004 & 0.011 & 0.012 \\ 0.008 & 0.012 & 0.015 \\ \end{pmatrix}$ \\ 
    33 & $\begin{pmatrix}0.39 \\ 0.28 \\ -0.74 \\ \end{pmatrix}$ & $\begin{pmatrix}0.016 & 0.017 & 0.011 \\ 0.017 & 0.022 & 0.013 \\ 0.011 & 0.013 & 0.008 \\ \end{pmatrix}$ \\ 
    34 & $\begin{pmatrix}0.40 \\ -0.37 \\ -0.76 \\ \end{pmatrix}$ & $\begin{pmatrix}0.007 & 0.007 & 0.004 \\ 0.007 & 0.010 & 0.007 \\ 0.004 & 0.007 & 0.004 \\ \end{pmatrix}$ \\ 
    35 & $\begin{pmatrix}0.41 \\ 0.32 \\ -0.86 \\ \end{pmatrix}$ & $\begin{pmatrix}0.014 & 0.016 & 0.017 \\ 0.016 & 0.019 & 0.021 \\ 0.017 & 0.021 & 0.023 \\ \end{pmatrix}$ \\ 
    36 & $\begin{pmatrix}0.46 \\ 0.39 \\ 0.88 \\ \end{pmatrix}$ & $\begin{pmatrix}0.010 & 0.007 & 0.011 \\ 0.007 & 0.018 & 0.014 \\ 0.011 & 0.014 & 0.016 \\ \end{pmatrix}$ \\ 
    37 & $\begin{pmatrix}0.48 \\ 0.51 \\ -0.07 \\ \end{pmatrix}$ & $\begin{pmatrix}0.010 & 0.011 & 0.007 \\ 0.011 & 0.012 & 0.006 \\ 0.007 & 0.006 & 0.004 \\ \end{pmatrix}$ \\ 
    38 & $\begin{pmatrix}0.55 \\ -0.12 \\ 0.72 \\ \end{pmatrix}$ & $\begin{pmatrix}0.009 & 0.009 & 0.008 \\ 0.009 & 0.011 & 0.005 \\ 0.008 & 0.005 & 0.013 \\ \end{pmatrix}$ \\ 
    39 & $\begin{pmatrix}0.56 \\ 0.21 \\ 0.42 \\ \end{pmatrix}$ & $\begin{pmatrix}0.014 & 0.012 & 0.006 \\ 0.012 & 0.013 & 0.007 \\ 0.006 & 0.007 & 0.004 \\ \end{pmatrix}$ \\ 
    40 & $\begin{pmatrix}0.66 \\ -0.28 \\ 0.41 \\ \end{pmatrix}$ & $\begin{pmatrix}0.012 & 0.010 & 0.013 \\ 0.010 & 0.010 & 0.011 \\ 0.013 & 0.011 & 0.016 \\ \end{pmatrix}$ \\ 
    41 & $\begin{pmatrix}0.69 \\ -0.68 \\ 0.12 \\ \end{pmatrix}$ & $\begin{pmatrix}0.014 & 0.010 & 0.004 \\ 0.010 & 0.009 & 0.006 \\ 0.004 & 0.006 & 0.007 \\ \end{pmatrix}$ \\ 
    42 & $\begin{pmatrix}0.70 \\ 0.77 \\ 0.53 \\ \end{pmatrix}$ & $\begin{pmatrix}0.016 & 0.012 & 0.016 \\ 0.012 & 0.009 & 0.012 \\ 0.016 & 0.012 & 0.018 \\ \end{pmatrix}$ \\ 
    43 & $\begin{pmatrix}0.73 \\ 0.71 \\ 0.62 \\ \end{pmatrix}$ & $\begin{pmatrix}0.015 & 0.009 & 0.014 \\ 0.009 & 0.013 & 0.009 \\ 0.014 & 0.009 & 0.013 \\ \end{pmatrix}$ \\ 
    44 & $\begin{pmatrix}0.75 \\ 0.94 \\ 0.74 \\ \end{pmatrix}$ & $\begin{pmatrix}0.012 & 0.010 & 0.013 \\ 0.010 & 0.009 & 0.011 \\ 0.013 & 0.011 & 0.020 \\ \end{pmatrix}$ \\ 
    45 & $\begin{pmatrix}0.79 \\ -0.34 \\ 0.64 \\ \end{pmatrix}$ & $\begin{pmatrix}0.018 & 0.010 & 0.013 \\ 0.010 & 0.007 & 0.009 \\ 0.013 & 0.009 & 0.016 \\ \end{pmatrix}$ \\ 
    46 & $\begin{pmatrix}0.81 \\ -0.86 \\ 0.35 \\ \end{pmatrix}$ & $\begin{pmatrix}0.016 & 0.014 & 0.016 \\ 0.014 & 0.013 & 0.014 \\ 0.016 & 0.014 & 0.017 \\ \end{pmatrix}$ \\ 
    47 & $\begin{pmatrix}0.81 \\ -0.85 \\ -0.45 \\ \end{pmatrix}$ & $\begin{pmatrix}0.014 & 0.013 & 0.014 \\ 0.013 & 0.014 & 0.015 \\ 0.014 & 0.015 & 0.017 \\ \end{pmatrix}$ \\ 
    48 & $\begin{pmatrix}0.91 \\ 0.54 \\ -0.75 \\ \end{pmatrix}$ & $\begin{pmatrix}0.011 & 0.011 & 0.011 \\ 0.011 & 0.012 & 0.012 \\ 0.011 & 0.012 & 0.013 \\ \end{pmatrix}$ \\ 
    49 & $\begin{pmatrix}0.91 \\ -0.58 \\ 0.66 \\ \end{pmatrix}$ & $\begin{pmatrix}0.018 & 0.015 & 0.005 \\ 0.015 & 0.014 & 0.004 \\ 0.005 & 0.004 & 0.002 \\ \end{pmatrix}$ \\ 
    50 & $\begin{pmatrix}0.93 \\ 0.09 \\ 0.95 \\ \end{pmatrix}$ & $\begin{pmatrix}0.009 & 0.013 & 0.008 \\ 0.013 & 0.019 & 0.014 \\ 0.008 & 0.014 & 0.013 \\ \end{pmatrix}$ \\ 
    \hline 
    \end{tabular}
\end{minipage}
}%
\label{tab:experiments_dists_gaussian_cvar}
\end{table}

In addition to the visualizations of the single experiment example in Section~\ref{numerics:portfolioopt}, Figure~\ref{fig:portfolioopt_singleexample_details} shows the underlying true distribution $\mathbb{P}$ and the samples drawn from it (left panel). Further, the corresponding change in portfolio allocations is shown in the right panel. The center panel again visualizes how the parametrization of the ambiguity set changes over the iterations.
\begin{figure}[htbp]
    \centering
    \includegraphics[width=0.9\textwidth]{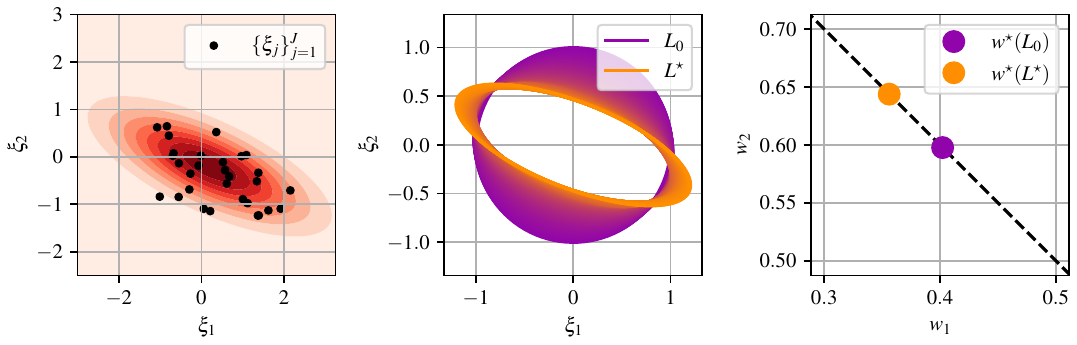}
    \caption{Underlying true distribution $\mathbb{P}$ (red contour) and samples (black), unit-cost ellipses defined by the matrices $L_i$ in the upper-level optimization, and optimal decisions before ($w^\star(L_0)$) and after ($w^\star(L^\star)$) optimization.}
    \label{fig:portfolioopt_singleexample_details}
\end{figure}



%
%
%
%
%
\subsection{Linear regression example}
\label{app:numerics_details:linreg}
We state the generation process used for the multi-experiments in Subsection~\ref{numerics:linear:regression}: $x$ in \eqref{eq:linreg_datagen_model} is uniformly sampled in the interval $[-10,10]$, $w$ is uniformly sampled the interval $[-10,10]$ and $\sigma^2$ is uniformly sampled in the interval $[500, 1000]$. This procedure leads to the $10$ underlying true models stated in Table~\ref{tab:linreg_models}.
\begin{table}[htbp]
    \small
    \centering
    \begin{adjustbox}{max width=\textwidth}
    \begin{tabular}{ccc}
    \hline
    Model & Weight ($w$) & Variance ($\sigma^2$) \\
    \hline
    1 & -6.7805 & 564.285 \\
    2 & -5.8464 & 625.412 \\
    3 & -2.7811 & 699.653 \\
    4 & -1.3851 & 710.190 \\
    5 & -0.0144 & 783.458 \\
    6 & 4.3483 & 846.372 \\
    7 & 6.3163 & 915.492 \\
    8 & 7.1061 & 922.537 \\
    9 & 8.5174 & 932.399 \\
    10 & 8.9350 & 978.001 \\
    \hline
    \end{tabular}
    \end{adjustbox}
    \vspace{0.2cm}
    \caption{Summary of the 10 models used for linear regression. For each model, \(x \sim \mathcal{U}(-10,10)\) and \(y = wx + e\) with \(e \sim \mathcal{N}(0,\sigma^2)\).}
    \label{tab:linreg_models}
\end{table}

Figure~\ref{fig:linreg_l_f_separate_1} presents the average worst-case and out-of-sample errors across all experiments in Subsection~\ref{numerics:linear:regression}. Notably, both error measures decrease after bilevel optimization (right) compared to before (left).

\begin{figure}[htbp]
    \centering
    \includegraphics[width=0.9\textwidth]{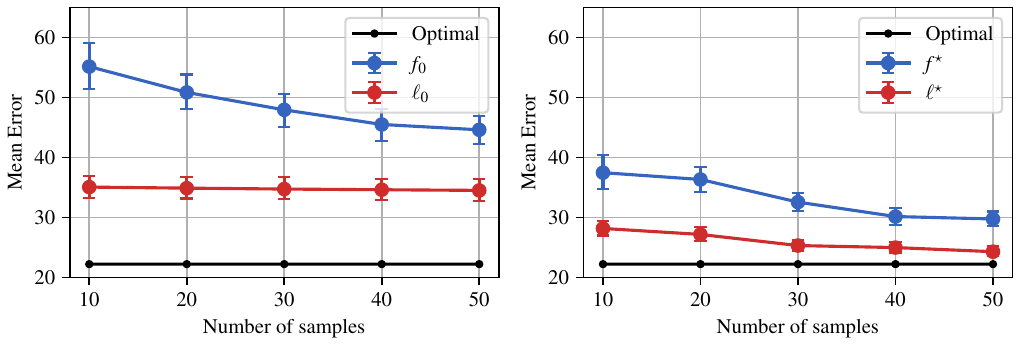}
    \caption{Average initial (left) and final (right) objectives (worst-case objective in blue and out-of-sample performance in red) of the distributionally robust linear regression model with absolute error over multiple experiments. The black line shows the average best possible expected error over all experiments.}
    \label{fig:linreg_l_f_separate_1}
\end{figure}

%% file: Contents/Appendix_E_additional_results.tex
This appendix provides additional results for different settings of the portfolio optimization and linear regression tasks.
\normalsize

\subsection{Portfolio optimization with Gaussian reference distribution (higher-dimensional case)}
\label{app:furtherresults:portfolioopt_gaussian_higherdim}

\mfnew{We investigate the distributionally robust portfolio optimization task with Gaussian reference distribution from Subsection~\ref{numerics:portfolioopt}, but now consider a higher number of assets, i.e., $k=10$. The means and covariances are generated using the same distributions as shown in Appendix~\ref{app:numerics_details:portfolioopt}. We generate $10$ independent experiments, each using a different true Gaussian distribution $\mathbb{P}$. For each distribution, $10$ distinct datasets are sampled, resulting in $100$ different trials for each sample size $J\in\{10,\ldots,100\}$. To evaluate the effectiveness of our procedure, we monitor the relative improvement of the worst-case objective and of the out-of-sample performance, as defined in Subsection~\ref{numerics:portfolioopt}.


Figure~\ref{fig:multiple_examples_true_gaussian_higherdim} illustrates the average improvement in both the worst-case objective and out-of-sample performance of the portfolio from before (left) to after (right) bilevel optimization. The displayed error bars represent bootstrapped confidence intervals, obtained by resampling the results to estimate the mean improvements. The results of this preliminary investigation seem to suggest that the advantages of our approach may scale favorably with dimension.

\begin{figure}[htbp]
    \centering
    \includegraphics[width=0.95\linewidth]{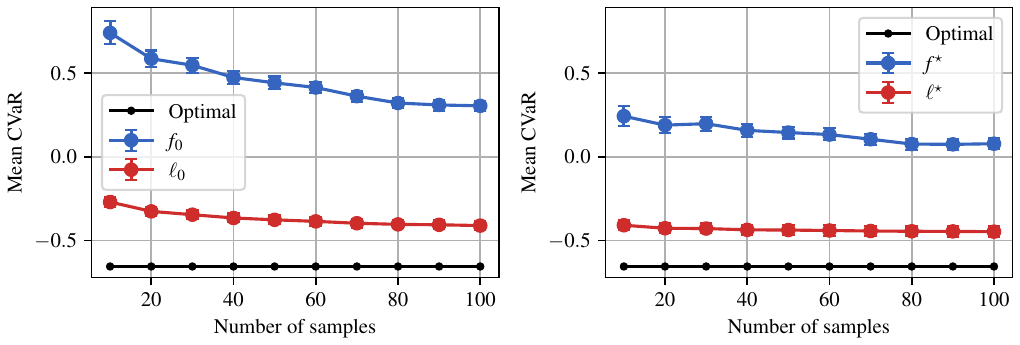}
    \caption{Average results for the higher-dimensional portfolio optimization example using a Gaussian reference distribution (worst-case objective in blue and out-of-sample performance in red) when considering multiple experiments with underlying true discrete distributions, before (left) and after (right) the bilevel optimization.}
    \label{fig:multiple_examples_true_gaussian_higherdim}
\end{figure}
}
\subsection{Portfolio optimization with empirical reference distribution}
\label{app:furtherresults:portfolioopt}
We investigate the distributionally robust portfolio optimization task when employing the empirical distribution
$
    \hat{\mathbb{P}} = \frac{1}{J} \sum_{j=1}^J \delta_{\hat{\xi}_j}
$
as the center of the optimal transport ambiguity set. 
To assess the results, we again consider the same metrics described in Subsection~\ref{numerics:portfolioopt}.
\subsubsection{Discrete true distribution}
\label{app:furtherresults:portfolioopt:discrete}
We generate $10$ true discrete distributions defined as
\begin{equation}
    \mathbb{P} = \sum_{i=1}^{10} p_i \delta_{x_i}
\end{equation}
where the $10$ different support points $x_i \in \mathbb{R}^k$ with $k = 3$ are uniformly sampled from the cube $[-1,1]^3$, and the weights $p_i$ are drawn from a Dirichlet distribution. For each true distribution, we generate 10 independent datasets via resampling. This entire process is repeated for $J \in \{10,...,100\}$, keeping the underlying true distributions fixed. We use the type-$1$ Mahalanobis distance to parametrize the transportation cost.

Figure~\ref{fig:multiple_examples_true_discrete} illustrates the average improvement in both the worst-case objective and out-of-sample performance of the portfolio from before (left) to after (right) bilevel optimization. The displayed error bars represent bootstrapped confidence intervals, obtained by resampling the results to estimate the mean improvements.

\begin{figure}[htbp]
    \centering
    \includegraphics[width=0.95\linewidth]{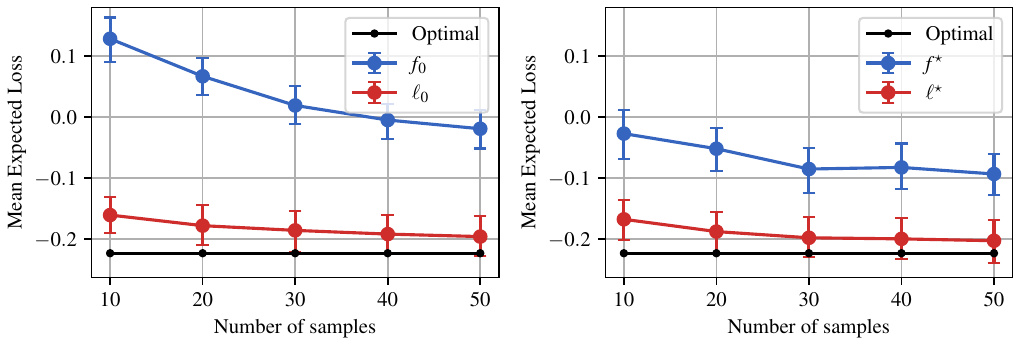}
    \caption{Average results (worst-case objective in blue and out-of-sample performance in red) when considering multiple experiments with underlying true discrete distributions, before (left) and after (right) the bilevel optimization.}
    \label{fig:multiple_examples_true_discrete}
\end{figure}

\subsubsection{Gaussian mixture model}
\label{app:furtherresults:portfolioopt:gmm}
We consider the case where the underlying true distribution is a Gaussian Mixture Model (GMM), defined as
\begin{equation}
    \mathbb{P} = \sum_{i=1}^3 \alpha_i \cdot \mathcal{N}(\mu_i, \Sigma_i),
\end{equation}
where the means $\mu_i$ are uniformly sampled in the cube $[-1,1]^3$ and the covariances $\Sigma_i$ are generated following \eqref{eq:true_gaussian_cov_generation}. The weights $\alpha_i$ are sampled from a Dirichlet distribution. Again, each distribution $\mathbb{P}$ is resampled $10$ times to generate independent datasets per experiment. This procedure is repeated for sample sizes $J \in \{10,...,100\}$, using the same underlying true distributions.

Figure~\ref{fig:multiple_examples_true_gmm} illustrates the average improvement in both the worst-case objective and out-of-sample performance of the portfolio from before (left) to after (right) bilevel optimization.
\begin{figure}[htbp]
    \centering
    \includegraphics[width=0.95\linewidth]{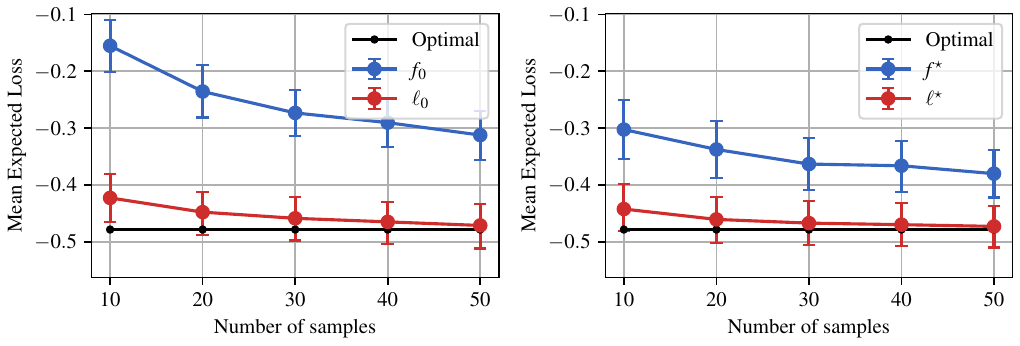}
    \caption{Average results (worst-case objective in blue and out-of-sample performance in red) when considering multiple experiments with underlying true GMM distributions, before (left) and after (right) the bilevel optimization.}
    \label{fig:multiple_examples_true_gmm}
\end{figure}


\subsection{Linear regression with squared error}
\label{app:furtherresults:linreg:squared}
We consider a linear regression task with squared loss function $\ell_2(w,\xi) = \left((-w,1)^\top \xi\right)^2$, and we use the type-$2$ Mahalanobis distance as a parametrization for the transportation cost. As in the $\ell_1$ error case reported in Subsection~\ref{numerics:linear:regression}, we first show the results of a single problem instance.
We consider a linear model corrupted by zero-mean Gaussian noise
\begin{equation}
    y = w x + \mathrm{e}, \text{ where } \mathrm{e} \sim \mathcal{N}(0,\sigma) \text{ and } x\sim \mathcal{U}(-10.0,10.0). \label{eq:linreg_datagen_model:appendix}
\end{equation}
The weight $w$ is deterministic and set to $1$, and the standard deviation of the noise is $\sigma=10$.

Figure~\ref{fig:linreg_single_example_squarederr} shows the true distribution of $\xi$ and the $J=20$ samples from it (left), the unit transportation cost ellipses across iterations (center), and the effect of our procedure on the linear model (right).
\begin{figure}[htbp]
    \centering
    \includegraphics[width=\textwidth]{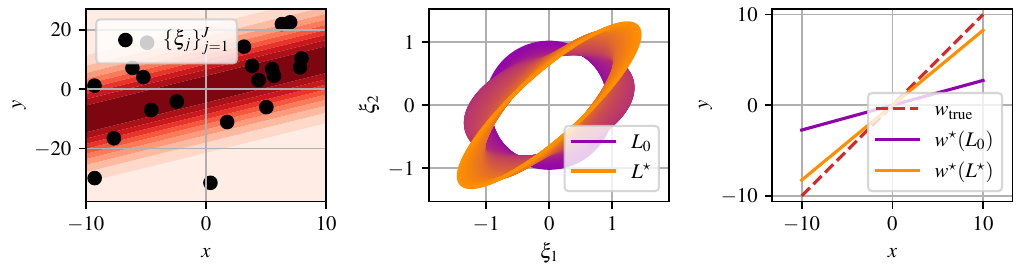}
    \caption{Underlying true distribution $\mathbb{P}$ (red contour) and samples (black), unit-cost ellipses defined by the matrices $L_i$ in the upper-level optimization, and optimal decisions before ($w^\star(L_0)$) and after ($w^\star(L^\star)$) optimization when using the squared error $\ell_2$. We set $n_b = 10$.}
    \label{fig:linreg_single_example_squarederr}
\end{figure}

Figure~\ref{fig:linreg_wc_true_improvements_squarederror} (left) shows the expected absolute error with respect to the worst-case distribution $\mathbb{Q}^\star(L)$, i.e., $\mathrm{e}_{\mathrm{wc}}(L) = \mathbb{E}_{\xi\sim\mathbb{Q}^\star(L)}\left( \ell_2(w^\star(L),\xi) \right) $ which corresponds to the upper level objective, and the expected absolute error on the true data generating process $\mathbb{P}$ approximated with $10^7$ samples, i.e., $\mathrm{e}_{\mathrm{oos}}(L) = \mathbb{E}_{\xi\sim\mathbb{P}}\left( \ell_2(w^\star(L),\xi) \right)$. Both metrics decrease across iterations, indicating a reduction in conservatism of the decision.
Additionally, we carry out a multi-experiment analysis by adopting the same models described in Appendix~\ref{app:numerics_details:linreg}. To evaluate the results, we again monitor the evolution of the relative improvement of the worst-case objective
\begin{equation}
    f_0 = \mathbb{E}_{\xi\sim\mathbb{Q}^\star(L_0)}\left[ \left( \bar{w}^\star(L_0)^\top \xi \right)^2 \right] \text{ and } f^\star = \mathbb{E}_{\xi\sim\mathbb{Q}^\star(L^\star)}\left[ \left( \bar{w}^\star(L^\star)^\top \xi \right)^2 \right] \nonumber
\end{equation}
and of the out-of-sample performance
\begin{equation}
    \ell_0 = \frac{1}{n_\mathrm{oos}} \sum_{i=1}^{N_\mathrm{oos}} \left( \bar{w}^\star(L_0)^\top \xi_i \right)^2 \text{ and } \ell^\star = \frac{1}{n_\mathrm{oos}} \sum_{i=1}^{N_\mathrm{oos}} \left( \bar{w}^\star(L^\star)^\top \xi_i \right)^2, \nonumber
\end{equation}
with $n_\mathrm{oos}=10^6$, across iterations. 
Figure~\ref{fig:linreg_wc_true_improvements_squarederror} (right) and Figure~\ref{fig:linreg_l_f_separate_2} jointly indicate a reduction in the conservatism of the DRO solution induced by the proposed bilevel procedure, corroborating the results discussed in Subsection~\ref{numerics:linear:regression}.

\begin{figure}[htbp]
    \centering
    \includegraphics[width=\linewidth]{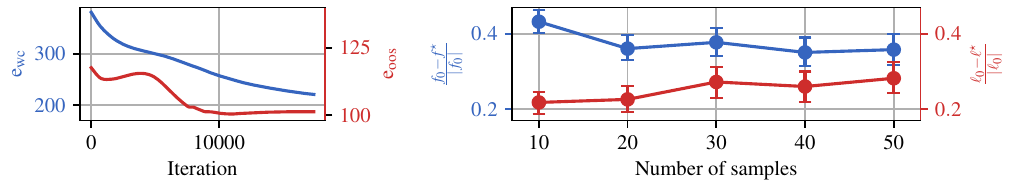}
    \caption{Expected squared error on the worst-case distribution $e_{\text{wc}}$ and expected squared error on the true data generating process $e_{\text{oos}}$ (left). Average improvement over multiple experiments w.r.t. the number of samples (right).}
    \label{fig:linreg_wc_true_improvements_squarederror}
\end{figure}




\begin{figure}[htbp]
    \centering
    \includegraphics[width=0.95\textwidth]{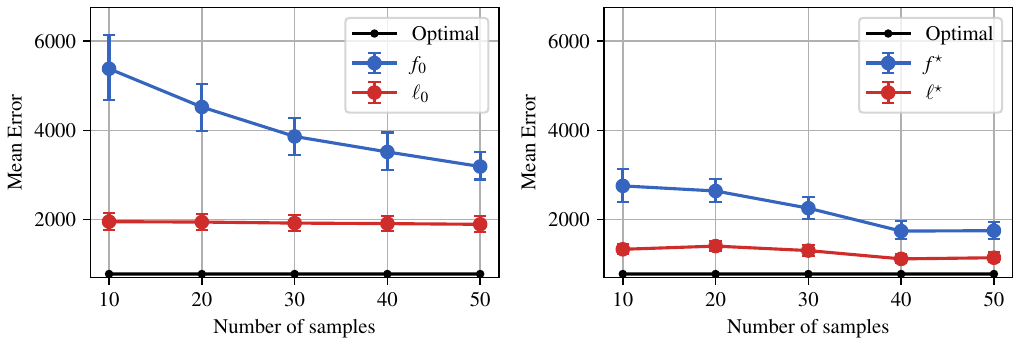}
    \caption{Average initial (left) and final (right) objectives (worst-case objective in blue and out-of-sample performance in red) of the distributionally robust linear regression model with squared error over multiple experiments. The black line shows the average best possible expected error over all experiments.}
    \label{fig:linreg_l_f_separate_2}
\end{figure}

\vfill

%% file: Contents/Complexity_and_Sensitivity.tex
\subsection{Computational Complexity of Differentiation Procedure}

\rzbegin

We analyze the complexity of the differentiation procedure described in Appendix \ref{app:diff_ot}.

In general, computing the hypergradient requires the differentiation of two terms: i) the primal-dual solution map of a conic program, and ii) the penalty term. As we outlined in Appendix \ref{app:diff_through_conic_programs}, differenting the solution of a conic program involves a set of matrix operations and the resolution of a linear system of equations (compare \cref{eq:app:differentiating:forward} for the forward derivative, and \cref{eq:app:differentiating:backward} for the adjoint), whose dimension equals $n+2m$, where $n$ denotes the dimension of the primal variable $x$, and $m$ is the dimension of the dual vector $y$ and the slack variable $s$. Generally, solving the linear system is the most computationally intensive operation, scaling with the cube of the dimension (for example when the system is solved using Gaussian elimination).

The number $J$ of samples may affect the number of constraints in the convex reformulation of the DRO problem, and potentially also the dimension of the primal variable. This means that more samples lead to a linear system of larger dimension and potentially to a greater computational complexity. Specifically, the primal decision variable $x$ scales linearly with the number $J$ of samples for the type-$1$ Mahalanobis distance in the regression example, and similarly in the type-$2$ Mahalanobis distance in both the portfolio optimization example and in the regression example. The slack variable $s$ and the dual variable $y$ both scale linearly with $J$ in the regression example with type-$2$ Mahalanobis distance, where the dependency is increased to $2J$ on the same example if the chosen distance is type-$1$ Mahalanobis. Moreover, these variables scale with $(d+2)^J$ in the case of type-$2$ Mahalanobis distance on the portfolio example. In all other examples, the sample size does not affect the dimension of the variables.

Next, differentiating the penalty function requires the differentiation of $n_b$ distances, computed on $n_b$ different samples in each iteration. However, because of the $\max$ term in $\varphi_\text{p}$, this is only required whenever the current design $\theta_i$ does not meet the required confidence level and thus $e(\theta_i)> \beta$. Differentiating the Mahalanobis distance requires solving a linear program whose dimension scales linearly with the number of samples (see Appendix \ref{ch:discrete_param_wassdist}). Differentiating the Gelbrich distance can be done by solving a set of $d^2$ Lyapunov equations (compare Appendix \ref{ch:gaussian_param_wassdist}), each scaling linearly with dimension $d$ of $\theta$. \rzend

\subsection{Sensitivity Analysis of Penalty Parameters}

\rzbegin We examine the influence of the penalty parameters $\lambda_\text{p}$ and $\eta_\text{p}$ from \cref{eq:bilevel_solvable,eq:penalty} on both the achieved performance improvement and adherence to the coverage constraint \cref{eq:constraint:boot}. For this analysis, we focus on the portfolio optimization problem described in \cref{numerics:portfolioopt}, setting the dimension to $k=10$ and utilizing $J=50$ samples. We use $n_{b}=20$ bootstrapped distributions and choose $\beta=0.1$ and $\gamma=0.05$. We execute the optimization algorithm multiple times ($10$ experiments with different underlying distributions, each experiment repeated $10$ times) across varying values of $\lambda_\text{p}$ and $\eta_\text{p}$.

As shown in \cref{fig:app:sensitivity:gaussian_dim_10}, the mean relative improvement remains consistently high across a broad spectrum of $(\lambda_\text{p}, \eta_\text{p})$ configurations. This consistency suggests that the algorithm’s performance is largely insensitive to the exact choice of penalty parameters, underscoring its practical robustness. \rzend

\begin{figure}[!ht]
\centering
\includegraphics[width=\linewidth]{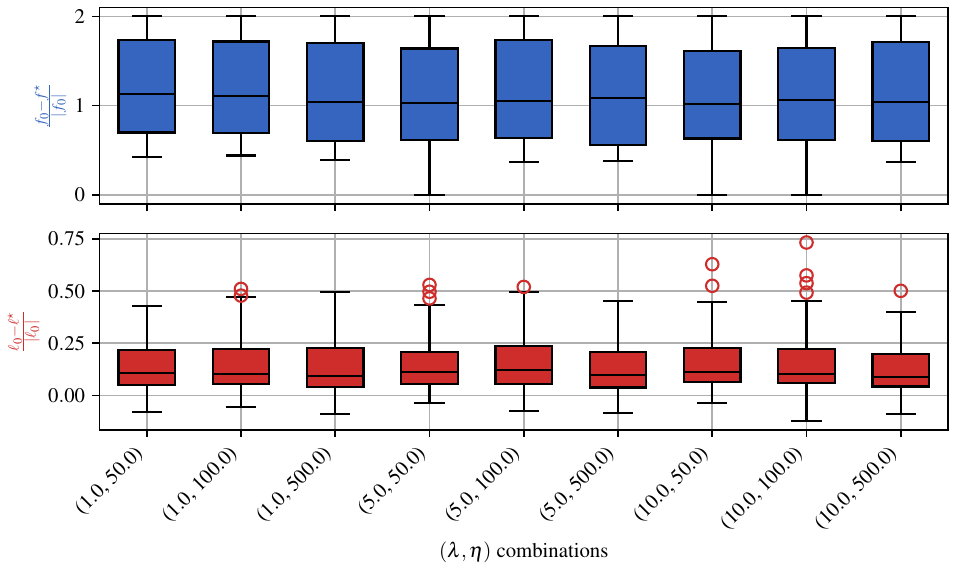}
\caption{Mean relative improvement for different combinations of $(\lambda_\text{p},\eta_\text{p})$.} \label{fig:app:sensitivity:gaussian_dim_10}
\end{figure}

\rzbegin Furthermore, \cref{fig:app:sensitivity:gaussian_dim_10_violation} shows the values of the coverage constraint \cref{constraint:upper} violation after reformulation across different parameter configurations. We observe that the constraint is consistently satisfied for all tested combinations of $(\lambda_\text{p}, \eta_\text{p})$, with only minor violations occurring. While smaller values of the penalty parameter (e.g., $\lambda_\text{p} = 1$) lead to slightly higher values of the constraint expression, they still remain within acceptable bounds, indicating that the penalization in \cref{eq:bilevel_solvable} is sufficient even in those cases. \rzend

\begin{figure}[!ht]
\centering
\includegraphics[width=\linewidth]{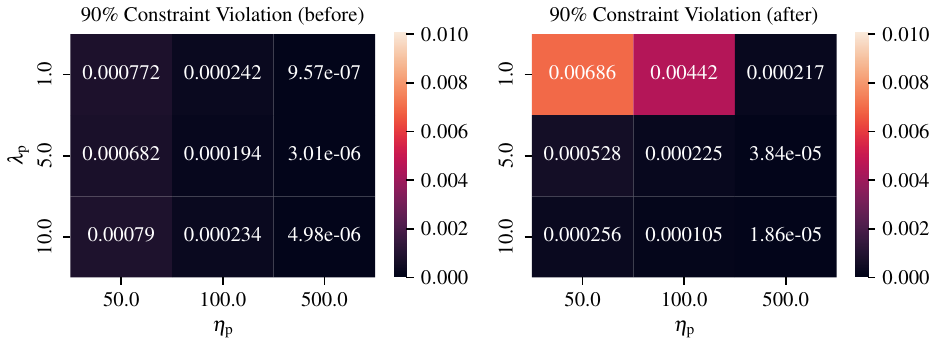}
\caption{Coverage constraint \cref{constraint:upper} violations (90th percentile) for different combinations of $(\lambda_\text{p},\eta_\text{p})$ over the $100$ experiments.} \label{fig:app:sensitivity:gaussian_dim_10_violation}
\end{figure}

\rzbegin These findings confirm that the proposed algorithm demonstrates strong robustness with respect to a wide and meaningful range of penalty parameter choices. Consequently, it can be confidently applied in safety-critical scenarios where hyperparameter tuning may be restricted or infeasible. \rzend